\documentclass{amsart}
\usepackage[utf8]{inputenc}
\usepackage{amsmath}
\usepackage{amsthm}
\usepackage{amsfonts}
\usepackage[alphabetic]{amsrefs}
\usepackage{geometry}
\usepackage{amscd}
\usepackage{graphicx}
\usepackage{enumitem}
\usepackage{amssymb}
\usepackage[all]{xy}
\usepackage{hyperref}
\usepackage{tikz}
\usepackage{tkz-graph}
\usepackage{color}
\usepackage{subfig}

\usepackage{rotating}
\usepackage{tabularx}
\usepackage{caption}
\usepackage{xspace}
\newcommand{\re}{\mathbb{R}}
\newcommand{\zz}{\mathbb{Z}}
\newcommand{\gf}{\mathbb{F}}
\newcommand{\gft}{\gf_2}
\newcommand{\evec}{\mathbf{e}}
\DeclareMathOperator{\degr}{deg}

\DeclareMathOperator{\Cayley}{Cayley}

\newtheorem{definition}{Definition}
\newtheorem{theorem}{Theorem}
\newtheorem{lemma}[theorem]{Lemma}
\newtheorem{corollary}[theorem]{Corollary}
\newtheorem{proposition}[theorem]{Proposition}
\newtheorem{example}{Example}
\newtheorem{conjecture}{Conjecture}

\newcommand{\ones}{\mathbf{1}}

\begin{document}

\title{Eigenvalues and Critical Groups of Adinkras}

\author{Kevin Iga\address{Natural Sciences Division,  Pepperdine University}\email{kevin.iga@pepperdine.edu}}
\author{Caroline Klivans\address{Division of Applied Mathematics,  Brown University}\email{klivans@brown.edu}}
\author{Jordan Kostiuk\address{Mathematics Department, Brown University }  \email{jordan\_kostiuk@brown.edu}}
\author{Chi Ho Yuen\address{Department of Mathematics, University of Oslo} \email{chihy@math.uio.no}}

\date{January 2022}

\maketitle
\begin{abstract}
Adinkras are signed graphs used to study supersymmetry in physics. We provide an introduction to these objects, and study the properties of their signed adjacency and signed Laplacian matrices. These matrices each have exactly two distinct eigenvalues (of equal multiplicity), making Adinkras closely related to the notions of strongly regular graphs.  We also study the critical groups of Adinkras, and in particular determine their odd components. A novel technique of independent interest is used which considers critical groups over  polynomial rings.
\end{abstract}

\tableofcontents

\section{Introduction}
\label{sec:intro}
Given an undirected simple graph, the \emph{discrete, or graph Laplacian} is an associated real symmetric matrix, inspired by the more famous analytic Laplacian on manifolds.  The graph Laplacian has been of great interest in areas such as numerical analysis, combinatorics, and clustering \cites{patodi1996riemannian, friedman1998computing, spielman1996spectral}.  
Being a real symmetric matrix, the graph Laplacian necessarily has real eigenvalues. 
The corresponding spectrum--the collection of eigenvalues with multiplicities--captures many interesting features of the graph.  For a survey, see \cite{Mohar}.

The adjacency matrix for a graph is also an integer symmetric matrix, and is related to the graph Laplacian, being the negative of the off-diagonal part of the graph Laplacian.  This adjacency matrix likewise has a spectrum which relates to various properties of the graph \cite{BelCioKooWang}.  One example is that if there are only two distinct eigenvalues of the adjacency matrix of a connected graph, then the graph must be a complete graph \cite{CvetRowlSim}.

As an integer matrix, the graph Laplacian also maps an integer lattice onto itself, and has a cokernel in this lattice.  This cokernel (or its torsion part) is called the critical group, sandpile group, Jacobian, or Picard group, depending on the context.  There is a large literature on investigating these groups for various classes of graphs and extensions beyond simple graphs, see e.g. \cites{RT_Signed, BKR_ChipFiring, criticalsimplicial, sandpiletree, criticalrep, Bai_hypercube}.   The critical group is a finite abelian group and understanding its structure, i.e. its invariant factors or, equivalently, its Smith Normal Form has also proven to be a very hard problem in general.  Much work has been done on the case of hypercubes and their generalizations, see e.g. \cites{Bai_hypercube, CSX_hypercube, AP_hypercube, jacobson2003critical}.

Signed graphs, which are graphs together with a labeling of each edge with either $+1$ or $-1$, were introduced by F. Harary in 1953, \cite{Harary} with further development by T. Zaslavsky \cites{zaslavsky,zaslavsky2}.  There is an analogous signed adjacency and signed Laplacian for signed graphs, and situations where there are only two eigenvalues have been of interest in recent years.\cites{GhasFath,BelCioKooWang,Ramezani,HouTangWang,Stanic2,Stanic3}  Indeed, signed graphs with two eigenvalues relates to the recent celebrated proof of the Sensitivity Conjecture \cite{Huang}.

In 2004, theoretical physicists M. Faux and S. J. Gates developed a computational tool, called an {\em Adinkra}, to help classify supersymmetric multiplets \cite{rA}.  An Adinkra is a graph with some additional structure, satisfying various conditions. As ordinary graphs, they are quotients of hypercubes by a doubly even code \cites{r6-codes,at0,doranCodesSupersymmetryOne2011,doranApplicationCubicalCohomology2017}. One of the additional structures of an Adinkra is that it is a signed graph.

In this paper, we introduce Adinkras and demonstrate that Adinkras are examples of strongly regular signed graphs \cites{Ramezani2,Stanic3,zaslavsky2}, and in particular, both the signed adjacency and signed Laplacian matrices each have precisely two eigenvalues. They therefore provide a large number of examples of signed graphs with this property.

Similar to studying signed Laplacians, the extension of critical groups to signed graphs is natural, and for Adinkras, there is a lot that can be determined.  We compute the odd component of the critical groups of any Adinkras, and completely determine the group for a large class of them. A key, novel strategy is to use the chromatic information of Adinkras to lift their Laplacians over polynomial rings, and consider the parallel problems there.

We begin with a review of signed graphs, as well as their signed adjacency and signed Laplacian matrices in Section~\ref{sec:signedgraph}.  We review Adinkras in Section~\ref{sec:adinkra}, and include a discussion on doubly even codes because all Adinkras can be constructed from them, using a quotient construction and disjoint unions.

In Section~\ref{sec:eigadinkra}, we show that Adinkras are examples of a number of phenomena in the literature, such as having two distinct eigenvalues (for both the signed adjacency and the signed Laplacian), and for being a strongly regular signed graph in several related senses.

Section~\ref{sec:criticalgroup} includes the invariant factors for small Adinkras $(N \leq 8, k \leq 4)$ organized via a classification in terms of $(N,k)$ doubly even codes.  Theorems~\ref{thm:inv_L}-\ref{thm:prism} prove properties of invariant factors for general Adinkras and determine the invariant factors in special cases.  

Next, in Section~\ref{sec:oddprime} we introduce formal variables $x_1,\ldots,x_N$ corresponding to the various colors of edges, and this gives us the invariant factors modulo odd primes.  As a side benefit, we prove a result about unsigned quotients of cubes, relating for example to the works \cites{jacobson2003critical, Bai_hypercube, GKM_Cayley} and may be of independent interest.

\section{Signed Graphs}
\label{sec:signedgraph}
We begin with a definition:
\begin{definition}
Given an undirected simple graph $G = (V,E)$, a \emph{signature} on $G$ is a map
\[\mu\colon E\to \{-1,1\}.\]
A \emph{signed graph} is a graph together with a signature.
\end{definition}

Signed graphs were introduced by F. Harary in 1953,\cite{Harary} and developed further by T. Zaslavsky  in \cite{zaslavsky,zaslavsky2}.

We will draw edges $e$ with $\mu(e)=1$ as solid edges, and edges with $\mu(e)=-1$ as dashed edges and use the terms \emph{solid edges} and \emph{dashed edges}, according to whether or not $\mu$ takes the value $1$ or $-1$ there.

\begin{example}
For instance, here is a signed graph:
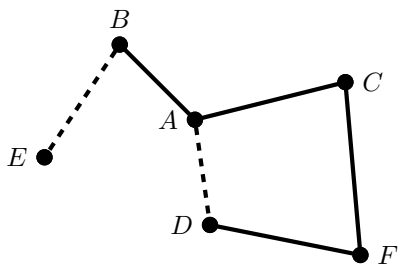
\begin{center}
\begin{tikzpicture}
\GraphInit[vstyle=Welsh]
\SetVertexNormal[MinSize=5pt]
\SetUpEdge[labelstyle={draw},style={ultra thick}]
\tikzset{Dash/.style={dashed,draw,ultra thick}}
\Vertex[x=0,y=0,Math,L={A},Lpos=180]{A}
\Vertex[x=-1,y=1,Math,L={B},Lpos=90]{B}
\Vertex[x=2,y=.5,Math,L={C},Lpos=0]{C}
\Vertex[x=.2,y=-1.4,Math,L={D},Lpos=180]{D}
\Vertex[x=-2,y=-.5,Math,L={E},Lpos=180]{E}
\Vertex[x=2.2,y=-1.8,Math,L={F},Lpos=0]{F}
\AddVertexColor{black}{A,B,C,D,E,F}
\Edge(A)(B)
\Edge(A)(C)
\Edge[style=Dash](A)(D)
\Edge[style=Dash](B)(E)
\Edge(C)(F)
\Edge(D)(F)
\end{tikzpicture}
\end{center}
Edges $AB$, $AC$, $CF$, and $DF$ are drawn solid to indicate that $\mu(AB)=\mu(AC)=\mu(CF)=\mu(DF)=1$.  Edges $BE$ and $AD$ are drawn dashed to indicate that $\mu(BE)=\mu(AD)=-1$.
\label{ex:signedgraph}
\end{example}

\subsection{The Signed Adjacency and Signed Laplacian Operators}

Given a finite signed graph and an ordering of its vertices, the signed adjacency operator is a symmetric integer-valued $\#V\times\#V$ matrix whose entries are:
\[A_{ij}=\begin{cases}
1,&\mbox{if there is a solid edge $e$ between vertices $v_i$ and $v_j$,}\\
-1,&\mbox{if there is a dashed edge $e$ between vertices $v_i$ and $v_j$, and}\\
0,&\mbox{if there is no such edge.}
\end{cases}
\]
Compare this with the (unsigned) adjacency matrix for an (unsigned) graph, which has its $(i,j)$ entry equal to $1$ whenever there is an edge $e$ between $v_i$ and $v_j$, and $0$ otherwise.  This is a special case of our definition here, where we take the signature as $+1$ for all edges.

Likewise, the signed Laplacian is also a symmetric integer-valued $\#V\times\#V$ matrix with entries:
\[L_{ij}=\begin{cases}
\degr(v_i),&\mbox{if $i=j$,}\\
-1,&\mbox{if $i\not=j$ and there is a solid edge $e$ between vertices $v_i$ and $v_j$,}\\
1,&\mbox{if $i\not=j$ and there is a dashed edge $e$ between vertices $v_i$ and $v_j$, and}\\
0,&\mbox{otherwise.}
\end{cases}
\]
We can write
\[L=D-A\]
where $D$ is a diagonal matrix whose $i$th diagonal entry is the degree of $v_i$.

If we take an unsigned graph, and use the signature of $+1$ for all edges, $L$ is then the (unsigned) graph Laplacian.

\begin{example}
Consider the signed graph from Example~\ref{ex:signedgraph}.

If we order the vertices as $(A,B,C,D,E,F)$, then
the signed adjacency matrix $A$ is
\[\begin{bmatrix}
    0&1&1&-1&0&0\\
    1&0&0&0&-1&0\\
    1&0&0&0&0&1\\
    -1&0&0&0&0&1\\
    0&-1&0&0&0&0\\
    0&0&1&1&0&0
\end{bmatrix}
\]
and the signed Laplacian matrix $L$ is
\[\begin{bmatrix}
    3&-1&-1&1&0&0\\
    -1&2&0&0&1&0\\
    -1&0&2&0&0&-1\\
    1&0&0&2&0&-1\\
    0&1&0&0&1&0\\
    0&0&-1&-1&0&2
\end{bmatrix}
\]

\end{example}

\section{Adinkras}
\label{sec:adinkra}
\emph{Adinkras} are bipartite signed graphs with additional ``decorations'', subject to various constraints, as described in this section.  They were introduced by S.~J. Gates and M. Faux to study supersymmetry in one (temporal) dimension \cite{rA}.

To help motivate Adinkras, we first discuss the supermultiplets and supersymmetry in one dimension (time).

\begin{definition}
A \emph{supermultiplet} consists of a finite number of functions of time called {\em bosons}, $b_1(t),\ldots,b_p(t)$, a finite number of functions of time called {\em fermions}, $f_1(t),\ldots,f_q(t)$, and $N\ge 1$ linear differential operators $Q_1,\ldots,Q_N$, so that for each $j$, $Q_ib_j$ is one of the following:
\begin{equation}
    Q_i(b_j)=\begin{cases}
    f_k\\
    -f_k\\
    \frac{d}{dt}f_k\\
    -\frac{d}{dt}f_k
    \end{cases}\label{eqn:qb}
\end{equation}
where $f_k$ is some fermion.\footnote{More generally, in the subject of supersymmetry, $Q_i(b_j)$ can be a linear combination of such terms, but the theory of Adinkras focuses on situations where there is only one such term.}  Likewise, for each $j$, $Q_if_j$ is one of the following:
\begin{equation}
    Q_i(f_j)=\begin{cases}
    \sqrt{-1}b_k\\
    -\sqrt{-1}b_k\\
    \sqrt{-1}\frac{d}{dt}b_k\\
    -\sqrt{-1}\frac{d}{dt}b_k
    \end{cases}\label{eqn:qf}
\end{equation}
The $Q_i$ are furthermore required to satisfy the following relations.  For all $i$,
\begin{equation}
Q_iQ_i=\sqrt{-1}\frac{d}{dt}
\label{eqn:susyalg1}
\end{equation}   
and whenever $i\not=j$,
\begin{equation}
 Q_iQ_j=-Q_jQ_i
 \label{eqn:susyalg2}
\end{equation}
\end{definition}

The abstract set of generators $Q_1,\ldots,Q_N$ satisfying these relations is called the \emph{extended super-Poincar\'e algebra in one dimension}.  Thus, a supermultiplet can be considered a kind of representation of this algebra.

By (\ref{eqn:susyalg1}), if
\begin{equation}
Q_i(b_j)=\pm \left(\frac{d}{dt}\right)^af_k,\label{eqn:bosontofermion}
\end{equation}
then
\begin{equation}
Q_i(f_k)=\pm\sqrt{-1}\left(\frac{d}{dt}\right)^{1-a}b_j.\label{eqn:fermiontoboson}
\end{equation}
Thus, for each $i$, $Q_i$ pairs each boson with a unique fermion.  It follows that the numbers of bosons and fermions are equal.  Furthermore, if one equation has a minus sign, then so does the other and precisely one of these equations has the time derivative.

An Adinkra with $N\ge 1$ colors is a graph that encodes these relationships of a supermultiplet.  The construction is reminiscent of Cayley graphs for finitely generated groups.  We draw vertices for bosons and fermions; we draw bosons as open circles and fermions as filled circles.  We choose $N$ colors for edges; one for each $Q_i$.  For each relation (\ref{eqn:bosontofermion}),
\[Q_i(b_j)=\pm\left(\frac{d}{dt}\right)^a f_k\]
we draw an edge from $b_j$ to $f_k$, colored using color $i$.  The relation (\ref{eqn:fermiontoboson}) shows that the same correspondence holds for fermions.

The edge is directed (i.e. drawn with an arrow) from $b_j$ to $f_k$ if $a=0$ and from $f_k$ to $ b_j$ if $a=1$, and drawn with a solid line if the $+$ sign is used, and a dashed line if the $-$ sign is used.

To summarize, here are the four possibilities:
\begin{center}
\renewcommand{\arraystretch}{1.7}
    \begin{tabular}{lll}
    $Q_i(b_j)=f_k$&
    \begin{tikzpicture}
    \GraphInit[vstyle=Welsh]
    \SetVertexNormal[MinSize=5pt]
    \SetUpEdge[labelstyle={draw},style={ultra thick,->}]
    \tikzset{Dash/.style={dashed,draw,ultra thick,->}}
    \Vertex[x=0,y=0,Math,L={b_j},Lpos=180]{B}
    \Vertex[x=2,y=0,Math,L={f_k},Lpos=0]{F}
    \AddVertexColor{black}{F}
    \Edge(B)(F);
    \end{tikzpicture}&
    $Q_i(f_k)=\sqrt{-1}\frac{d}{dt}b_j$\\
    $Q_i(b_j)=-f_k$&
    \begin{tikzpicture}
    \GraphInit[vstyle=Welsh]
    \SetVertexNormal[MinSize=5pt]
    \SetUpEdge[labelstyle={draw},style={ultra thick,->}]
    \tikzset{Dash/.style={dashed,draw,ultra thick,->}}
    \Vertex[x=0,y=0,Math,L={b_j},Lpos=180]{B}
    \Vertex[x=2,y=0,Math,L={f_k},Lpos=0]{F}
    \AddVertexColor{black}{F}
    \Edge[style=Dash](B)(F);
    \end{tikzpicture}&
        $Q_i(f_k)=-\sqrt{-1}\frac{d}{dt}b_j$\\
    $Q_i(b_j)=\frac{d}{dt}f_k$&
    \begin{tikzpicture}
    \GraphInit[vstyle=Welsh]
    \SetVertexNormal[MinSize=5pt]
    \SetUpEdge[labelstyle={draw},style={ultra thick,->}]
    \tikzset{Dash/.style={dashed,draw,ultra thick,->}}
    \Vertex[x=0,y=0,Math,L={b_j},Lpos=180]{B}
    \Vertex[x=2,y=0,Math,L={f_k},Lpos=0]{F}
    \AddVertexColor{black}{F}
    \Edge(F)(B);
    \end{tikzpicture}&
            $Q_i(f_k)=\sqrt{-1}b_j$\\
    $Q_i(b_j)=\frac{d}{dt}f_k$&
    \begin{tikzpicture}
    \GraphInit[vstyle=Welsh]
    \SetVertexNormal[MinSize=5pt]
    \SetUpEdge[labelstyle={draw},style={ultra thick,->}]
    \tikzset{Dash/.style={dashed,draw,ultra thick,->}}
    \Vertex[x=0,y=0,Math,L={b_j},Lpos=180]{B}
    \Vertex[x=2,y=0,Math,L={f_k},Lpos=0]{F}
    \AddVertexColor{black}{F}
    \Edge[style=Dash](F)(B);
    \end{tikzpicture}&
    $Q_i(f_k)=-\sqrt{-1}b_j$\\
    \end{tabular}
\end{center}
By this construction, each vertex has $N$ edges incident to it: one of each color.

To satisfy (\ref{eqn:susyalg2}), we consider two distinct colors $i$ and $j$.  In order for $Q_iQ_j(b_k)=-Q_jQ_i(b_k)$, we must have two bosons and two fermions in a $4$-cycle, like this (where we have used black for $Q_i$ and red for $Q_j$):

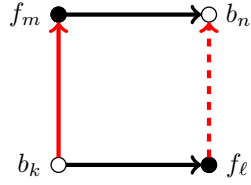
\begin{center}
    \begin{tikzpicture}
    \GraphInit[vstyle=Welsh]
    \SetVertexNormal[MinSize=5pt]
    \SetUpEdge[labelstyle={draw},style={ultra thick,->}]
    \tikzset{Dash/.style={dashed,draw,ultra thick,->}}
    \Vertex[x=0,y=0,Math,L={b_k},Lpos=180]{B1}
    \Vertex[x=2,y=0,Math,L={f_\ell},Lpos=0]{F1}      \Vertex[x=2,y=2,Math,L={b_n},Lpos=0]{B2}
    \Vertex[x=0,y=2,Math,L={f_m},Lpos=180]{F2}
    \AddVertexColor{black}{F1,F2}
    \Edge(B1)(F1);
    \Edge(F2)(B2);
    \Edge[color=red](B1)(F2);
    \Edge[color=red,style=Dash](F1)(B2);
    \end{tikzpicture}.
\end{center}

Compositions of several $Q$s give rise to a walk in the graph in the following way: if $v$ is a vertex, and $Q_{i_m}\cdots Q_{i_i}$ is a composition of $Q$s, then we get a sequence of vertices and edges: start with $v$.  Then the first edge is the edge of color $i_1$ incident to $v$.  The other vertex incident with this edge is the next vertex in the sequence.  In this way, we get a sequence of vertices and edges, where each edge connects the previous vertex to the next one; in other words, a walk.

In this example, the composition $Q_iQ_j(b_k)$ gives the walk that starts at $b_k$ at the lower left, then goes up to $f_m$, then right to $b_n$.  The composition $Q_jQ_i(b_k)$ starts at $b_k$ at the lower left, then goes right to $f_\ell$, then goes up to $b_n$.  The fact that these two walks end up at the same vertex is not an accident: it is a consequence of (\ref{eqn:susyalg2}): $Q_iQ_j(b_k)=-Q_jQ_i(b_k)$.  In other words, this is a $4$-cycle.  But in order to get the minus sign, there must be an odd number of dashed edges in this $4$-cycle.  Furthermore, each walk must involve the same number of derivatives, so there must be the same number of arrows directed along the walk for $Q_iQ_j(b_k)$ and $Q_jQ_i(b_k)$.\footnote{Unlike the usual treatment of directed graphs, a walk may go with or against the arrow.}

This last criterion is equivalent to the existence of an integer-valued function on the set of vertices $h\colon V\to\zz$, where for every arrow from vertex $v$ to vertex $w$, we have $h(w)=h(v)+1$.  This function can be used to reproduce the orientation of the edges, and so we can instead use an undirected graph, together with the function $h$, as long as vertices $v$ and $w$ that are adjacent have either $h(w)=h(v)+1$ or $h(v)=h(w)+1$.

This leads to the following definition, first found in \cite{r6-1}.  A more complete description of Adinkras for mathematicians is \cite{zhang}.

\begin{definition}
An \emph{Adinkra with $N$ colors} is a finite bipartite signed graph $(V,E,\mu)$, together with:
\begin{itemize}
    \item an assignment, to each edge, of one of $N$ colors;
    \item an integer-valued function on the set of vertices, i.e., $h\colon V\to \zz$;
\end{itemize}
subject to the following conditions:
\begin{itemize}
    \item each vertex is incident to $N$ edges, one of each color;
    \item the undirected subgraph consisting of edges of any two distinct colors is a disjoint union of $4$-cycles.  Such cycles will be called \emph{bicolor cycles};
    \item any such bicolor cycle contains an odd number of dashed edges $e$;
    \item if $v$ and $w$ are adjacent vertices then $|h(v)-h(w)|=1$.
\end{itemize}
We insist on Adinkras being non-empty and $N\ge 1$.  This is a minor restriction and will avoid filling the exposition with qualifiers.
\end{definition}

The bipartition of the graph comes from the fact that some vertices are bosons and others are fermions.

A signature $\mu$ is called \emph{totally odd} if it satisfies the third condition in the definition above (if all bicolor cycles have an odd number of dashed edges).

The existence of a function $h$ satisfying the fourth condition is equivalent to the graph being bipartite ($h\bmod{2}$ gives a bipartition, and conversely, given a bipartition, we can define $h(v)$ to be $0$ for bosons and $1$ for fermions).  There are typically many such functions $h$ on a bipartite graph, but since the specific function $h$ plays no role in this paper, so we ignore it from now on.

\begin{example}
The following diagram is an Adinkra with 3 colors: red, blue, and black.

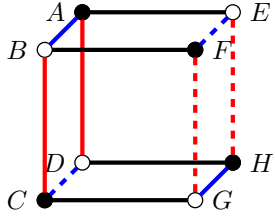
\begin{center}
\begin{tikzpicture}[scale=0.1]
\GraphInit[vstyle=Welsh]
\SetVertexNormal[MinSize=5pt]
\SetUpEdge[labelstyle={draw},style={ultra thick}]
\tikzset{Dash/.style={dashed,draw,ultra thick}}
\Vertex[x=0,y=0,Math,L={B},Lpos=180]{B}
\Vertex[x=20,y=0,Math,L={F},Lpos=0]{F}
\Vertex[x=0,y=-20,Math,L={C},Lpos=180]{C}
\Vertex[x=20,y=-20,Math,L={G},Lpos=0]{G}
\Vertex[x=5,y=5,Math,L={A},Lpos=180]{A}
\Vertex[x=25,y=5,Math,L={E},Lpos=0]{E}
\Vertex[x=5,y=-15,Math,L={D},Lpos=180]{D}
\Vertex[x=25,y=-15,Math,L={H},Lpos=0]{H}
\AddVertexColor{black}{F,C,A,H}
\Edge(A)(E)
\Edge(D)(H)
\Edge[color=red](A)(D)
\Edge[color=red,style=Dash](E)(H)
\Edge[color=blue](A)(B)
\Edge[color=blue,style=Dash](E)(F)
\Edge[color=blue,style=Dash](D)(C)
\Edge[color=blue](G)(H)
\Edge(B)(F)
\Edge(C)(G)
\Edge[color=red](B)(C)
\Edge[color=red,style=Dash](F)(G)
\end{tikzpicture}
\end{center}

As a graph, it consists of the vertices and edges of a cube in three dimensions.  As the reader can check, each vertex is incident to one red edge, one blue edge, and one black edge.  If we select any two colors, the result is two opposing square faces.  For instance, selecting red and blue edges gives the left and right faces of this cube.  The left face has one dashed edge, and the right face has three dashed edges.  As will be typical in this paper, we ignore the $h$ function.
\end{example}

\subsection{Hypercubes and Prisms}
An important class of Adinkras with $N$ colors come from $N$-dimensional hypercubes.  The set of vertices is $\{0,1\}^N$.  Two vertices are connected by an edge of color $i$ if the two vertices differ in only one coordinate.  Then the first two conditions in the definition of an Adinkra are satisifed.  The fact that the graph is bipartite can be seen by taking each vertex $(x_1,\ldots,x_N)$ and looking at $\sum x_j\pmod{2}$.

It is possible to choose a signature on these hypercubes so that the third condition is also satisfied.  In fact, there are many such choices.  The story is more fully explained in~\cite{doranApplicationCubicalCohomology2017}.  For our purposes, there is an inductive process of obtaining a totally odd signature on the $N+1$-cube from a totally odd signature on the $N$-cube.

Specifically, given an Adinkra $A$ with $N$ colors, we can create an Adinkra with $N+1$ colors called the \emph{prism on $A$}, written $A\times I$, as follows:

If the vertex set for $A$ is $V$, then the vertex set for $A\times I$ is $V\times\{0,1\}$.  There are three kinds of edges in $A\times I$:
\begin{enumerate}
    \item If $e$ is an edge in $A$ of color $j$ connecting $v$ to $w$, then $(e,0)$ is an edge in $A\times I$ of color $j$ connecting $(v,0)$ to $(w,0)$.  It is dashed if and only if $e$ is.
    \item If $e$ is an edge in $A$ of color $j$ connecting $v$ to $w$, then $(e,1)$ is an edge in $A\times I$ of color $j$ connecting $(v,1)$ to $(w,1)$. It is dashed if and only if $e$ is.
    \item If $v$ is a vertex in $A$, then there is an edge called $v*$ of color $N+1$ connecting $(v,0)$ to $(v,1)$.  It is solid if $v$ is a fermion and dashed if $v$ is a boson.
\end{enumerate}

It is straightforward to check that the result satisfies the conditions in the definition of an Adinkra with $N+1$ colors.  The prism of an $N$-cube is an $N+1$-cube, so this suffices to define the $N+1$-cube Adinkra.

\subsection{Doubly even codes}
\label{sec:codes}
A rich source of examples of Adinkras comes from quotienting these cubical examples with codes.  We begin with a review of codes.  A more thorough introduction to the subject is~\cite{rCHVP}.

We view $\{0,1\}^N=\gft^N$ as a vector space over the field $\gft$ of two elements.  Given an element of $\gft^N$, written $(x_1,\ldots,x_N)$, the \emph{weight} of the element is the number of $j$ with $x_j=1$.  
A \emph{linear code $C$} of length $N$ is any vector subspace of $\gft^N$.  We refer to the dimension of the vector subspace as the dimension of the code $C$.  We say a code is $[N,k]$ if it is of length $N$ and dimension $k$.
A linear code is called \emph{doubly even} if every element of the code has weight divisible by $4$.  The \emph{minimal weight} of a nontrivial code is the minimum weight among nonzero elements of the code.  The minimal weight is of importance in applications to finding error correcting and error detecting codes.  Of course, the minimal weight of a nontrivial doubly even code is a multiple of 4.

A linear code, being a vector subspace, can be specified by describing a basis.  It is traditional in the coding theory literature to write these basis elements as rows of a matrix, called a \emph{generating matrix} for $C$.  The smallest nontrivial doubly even code is called $d_4$ and has length $N=4$ and dimension $1$.  Its generating matrix is
\[\begin{bmatrix}
    1&1&1&1
\end{bmatrix}\]
and the code is $\{0000,1111\}$.

The next doubly even code in size is of length $N=6$ and has dimension $2$, and is called $d_6$, with the following generating matrix:
\[\begin{bmatrix}
    1&1&1&1&0&0\\
    0&0&1&1&1&1
\end{bmatrix}.\]

More generally, for even $N\ge 4$, there is a doubly even code $d_N$ of length $N$ and dimension $\frac{N}{2}-1$, with the following generating matrix:
\[\left[\begin{array}{ccccccccccccc}
    1&1&1&1&0&0&0&\cdots&0&0&0&0&0\\
    0&0&1&1&1&1&0&\cdots&0&0&0&0&0\\
    \vdots&&&&&&&&&&&&\vdots\\
    0&0&0&0&0&0&0&\cdots&0&1&1&1&1
    \end{array}\right]\]
where each row is obtained from the row above it by shifting it two columns to the right.

There are other doubly even codes.  For instance, the code $e_7$ has length $7$ and has three generators:

\[\left[
\begin{array}{ccccccc}
1&1&1&1&0&0&0\\
0&0&1&1&1&1&0\\
1&0&1&0&1&0&1
\end{array}\right].
\]
The code $e_7$ is related to the famous Hamming $[7,4]$ code, which was the founding example for error correcting codes: if you restrict your messages to be in this subspace, then any flip of a single bit can be corrected because it is one bit away from a unique element of the subspace, \cite{Hamming}.
The relationship between the codes comes from the fact that the Hamming $[7,4]$ code is the orthogonal complement of the $e_7$ code with respect to the inner product:
\[\langle (v_1,\ldots,v_N),(w_1,\ldots,w_N)\rangle= \sum_j v_jw_j\pmod{2}.\]
Note that the Hamming $[7,4]$ code itself is not doubly even.

The code $e_8$ has length $8$ and four generators:

\[\left[
\begin{array}{cccccccc}
1&1&1&1&0&0&0&0\\
0&0&1&1&1&1&0&0\\
0&0&0&0&1&1&1&1\\
1&0&1&0&1&0&1&0
\end{array}\right].
\]
It is equivalent to the Extended Hamming $[8,4]$ code which is obtained from the Hamming $[7,4]$ code by adding a single bit to the right of each element in such a way  that every basis element has even weight. 

The notation $d_N$, $e_7$, and $e_8$ connects to lattices, root systems, and Lie algebras.  Given a linear code $C\subset \gft^N$, the set of integer points $(x_1,\ldots,x_N)\in \zz^N$ such that $(x_1,\ldots,x_N)\bmod{2}\in C$ forms a lattice.  Starting from the $d_N$, $e_7$, $e_8$ respectively, we get the lattices $D_N$, $E_7$, and $E_8$.  The nonzero elements of the lattice that are closest to the origin form a root system, and these are indeed the root systems $D_N$, $E_7$, and $E_8$ which are the root systems associated with the Lie algebras $D_N$, $E_7$, and $E_8$.  More generally, there is a relationship between root systems of Lie algebras, lattices, and codes, and these are special cases of it.\cite{rCHVP,rCS}

The code $t^j$ means the trivial code of length $j$ consisting of $\{0\cdots 0\}$.

Given codes $C_1$ and $C_2$, the code $C_1\oplus C_2$ (the direct sum of $C_1$ and $C_2$) means the set of all concatenations $(x_1,\ldots,x_i,y_1,\ldots,y_j)$ where $(x_1,\ldots,x_i)\in C_1$ and $(y_1,\ldots,y_j)\in C_2$.  If $C_1$ has length $N_1$ and dimension $k_1$, and $C_2$ has length $N_2$ and dimension $k_2$, then $C_1\oplus C_2$ has length $N_1+N_2$ and dimension $k_1+k_2$.

A special case is $C_1\oplus t^j$.  This is the set of all elements of $C_1$, with $j$ zeros appended to the end of each.

The classification of doubly even codes is difficult.  But there are formulas due to P. Gaborit counting the number of doubly even codes for each $N$ and $k$ \cite{Gaborit}; these formulas depend on $N\pmod{8}$.  The formulas are cumbersome but it may suffice to satisfy the reader's curiosity to have the following formula for the maximum dimension for a given $N$.  This follows from Gaborit's formulas in \cite{Gaborit}.  The reader can check that these maxima can be achieved with direct sums of $t^j$, $d_N$, $e_7$, and $e_8$ codes.
\begin{theorem}
Let $N=8m+p$, where $0\le p\le 7$ and $m$ are integers.  Then there exists a doubly even code of length $N$ and dimension $k$ if and only if
\begin{equation}
    k\le \begin{cases}
    4m,&p=0,1,2,3\\
    4m+1,&p=4,5\\
    4m+2,&p=6\\
    4m+3,&p=7.
    \end{cases}
\end{equation}
\label{thm:maxkcode}
\end{theorem}

The case where $N$ is a multiple of 8 (so that $p=0$) and $k$ is maximal (so that $k=N/2$) has historically been of special interest because these are self-dual codes.  It is also of interest in supersymmetry because four-dimensional $N\ge 2$ supersymmetry turns into one-dimensional $N$ a multiple of 8 when we ignore the spatial dimensions, in a process called \emph{dimensional reduction}.

For $N=8$, the unique maximal code ($k=4$) is $e_8$.  For $N=16$, there are two maximal codes: $e_8\oplus e_8$, and $e_{16}$, which consists of $d_{16}$ together with another generator that has $1$s on the odd coordinates, analogous to $e_8$:

\[\begin{bmatrix}
1111000000000000\\
0011110000000000\\
0000111100000000\\
0000001111000000\\
0000000011110000\\
0000000000111100\\
0000000000001111\\
1010101010101010
\end{bmatrix}\]

For $N=24$ there are 9 maximal codes ($k=12$) and one of them is particularly famous: the extended binary Golay code.  This is a code with minimal weight $8$, and in particular has no codewords of weight 4.  The corresponding lattice is the Leech lattice.  Here is a generating matrix:
\[\begin{bmatrix}
  100000000000100111110001\\
  010000000000010011111010\\
  001000000000001001111101\\
  000100000000100100111110\\
  000010000000110010011101\\
  000001000000111001001110\\
  000000100000111100100101\\
  000000010000111110010010\\
  000000001000011111001001\\
  000000000100001111100110\\
  000000000010010101010111\\
  000000000001101010101011
\end{bmatrix}
\]

For $N=32$ there are 85 maximal codes, 5 of which have minimal weight $8$.\cite{rBilRees,CPS32}  These are of interest because dimensional reduction for $N=1$ ten dimensional supersymmetry, used in superstring theory, results in $N=32$ supersymmetry in one dimension.

\subsection{Quotients by Codes}
\label{sec:quotient}
Given an $N$-cube Adinkra and a linear code $C$, we can define a graph called the quotient of the $N$-cube by $C$, whose vertex set is $\gft^N/C$, the quotient space by the subspace $C$, consisting of cosets of $C$.  An edge of color $j$ connects two such cosets $v+C$ and $w+C$ if there are vertices in $v+C$ and $w+C$ connected by an edge of color $j$ in $\gft^N$.  We write $[v]$ for $v+C$.

The following is straightforward to prove:
\begin{proposition}
If $[v]$ and $[w]$ are adjacent, then for every representative $v\in[v]$, there exists a $w\in[w]$ so that $v$ and $w$ are adjacent.
\label{prop:adjacentrep}
\end{proposition}

\begin{theorem}[\cite{doranApplicationCubicalCohomology2017}]
A totally odd signature exists on the quotient of an $N$-cube by the linear code $C$ if and only if $C$ is doubly even.
\end{theorem}

Thus, every doubly even linear code gives rise to an Adinkra.  If the code is of length $N$ and dimension $k$, then the Adinkra has $N$ colors and $2^{N-k}$ vertices.

One of the major results in the classification of Adinkras states that all connected Adinkras arise as quotients of hypercubes.  
\begin{theorem}[Theorem 4.3 \cite{doranCodesSupersymmetryOne2011}]
Every connected Adinkra is colored-graph isomorphic to a quotient of a hypercube Adinkra by a doubly even code.  More generally, every Adinkra is a disjoint union of such quotients.\label{thm:at}
\end{theorem}
The term {\em colored-graph isomorphic} here refers to a graph isomorphism that preserves colors.  But the signature is not assumed to be preserved.

Thus, to specify a connected Adinkra, it suffices to specify a code, and a totally odd signature.

\begin{theorem}
Given $N=8m+p$, where $0\le p\le 7$ and $m$ are integers, the minimum number of vertices for an Adinkra with $N$ colors is:
\begin{equation}
    \#V_{\min} = \begin{cases}
    2^{4m},&p=0\\
    2^{4m+1},&p=1\\
    2^{4m+2},&p=2\\
    2^{4m+3},&p=3,4\\
    2^{4m+4},&p=5,6,7
    \end{cases}\label{eqn:minadinkra}.
\end{equation}
Every Adinkra with $N$ colors has a number of vertices which is a multiple of this.
\label{thm:minadinkra}
\end{theorem}

\begin{proof}
If an Adinkra is connected, Theorem~\ref{thm:at} says that it is a quotient of an $N$-cube with a doubly even code.  The number of vertices is $2^{N-k}$.  By Theorem~\ref{thm:maxkcode}, we have a formula for the maximum dimension, which we call $k_{\max}$.  Then $2^{N-k_{\max}}$ divides $2^{N-k}=\#V$.  The formula for $2^{N-k_{\max}}$ can be worked out case-by-case according to $p$, and the result is (\ref{eqn:minadinkra}).

Every Adinkra is a disjoint union of connected Adinkras, so the divisibility condition continues to hold.
\end{proof}

\begin{corollary}
If $N\ge 2$, the number of vertices of an Adinkra with $N$ colors is a multiple of 4.
\label{cor:vmult4}
\end{corollary}

\begin{example}
\label{ex:k4}
The following figure is a quotient of the $4$-cube by the code $d_4=\{(0000),(1111)\}$.
The binary labeling used in the lower figure comes from considering the cosets of $\gft^4/d_4$.  Each coset is a pair of 4-tuples, which differ in every coordinate, for instance, $\{1011,0100\}$.  We label the coset by picking the representative from the coset whose last coordinate is $0$.  Then these form a hypercube according to the first three coordinates, and indeed, if we focus only on colors 1, 2, and 3, then this is the 3-cube.  But the fourth color connects antipodal vertices.

There are many totally odd signatures on this graph, and the signature in the drawing below is one such.

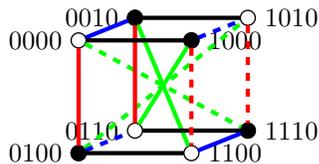
\begin{center}
\begin{tikzpicture}[scale=0.15]
\GraphInit[vstyle=Welsh]
\SetVertexNormal[MinSize=5pt]
\SetUpEdge[labelstyle={draw},style={ultra thick}]
\tikzset{Dash/.style={dashed,draw,ultra thick}}
\Vertex[x=0,y=0,Math,L={0000},Lpos=180]{0000}
\Vertex[x=10,y=0,Math,L={1000},Lpos=0]{1000}
\Vertex[x=0,y=-10,Math,L={0100},Lpos=180]{0100}
\Vertex[x=10,y=-10,Math,L={1100},Lpos=0]{1100}
\Vertex[x=5,y=2,Math,L={0010},Lpos=180]{0010}
\Vertex[x=15,y=2,Math,L={1010},Lpos=0]{1010}
\Vertex[x=5,y=-8,Math,L={0110},Lpos=180]{0110}
\Vertex[x=15,y=-8,Math,L={1110},Lpos=0]{1110}
\AddVertexColor{black}{1000,0100,0010,1110}
\Edge[color=green,style=Dash](0000)(1110)
\Edge[color=green](1000)(0110)
\Edge[color=green,style=Dash](0100)(1010)
\Edge[color=green](1100)(0010)

\Edge(0010)(1010)
\Edge(0110)(1110)
\Edge[color=red](0010)(0110)
\Edge[color=red,style=Dash](1010)(1110)
\Edge[color=blue](0000)(0010)
\Edge[color=blue,style=Dash](1000)(1010)
\Edge[color=blue,style=Dash](0100)(0110)
\Edge[color=blue](1100)(1110)
\Edge(0000)(1000)
\Edge(0100)(1100)
\Edge[color=red](0000)(0100)
\Edge[color=red,style=Dash](1000)(1100)
\end{tikzpicture}
\end{center}

We can rearrange these to show that we have a $K(4,4)$ graph:
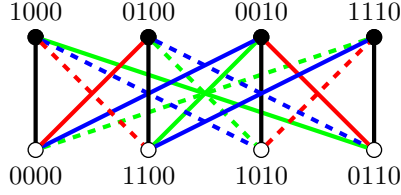
\begin{center}
\begin{tikzpicture}[scale=0.15]
\GraphInit[vstyle=Welsh]
\SetVertexNormal[MinSize=5pt]
\SetUpEdge[labelstyle={draw},style={ultra thick}]
\tikzset{Dash/.style={dashed,draw,ultra thick}}
\Vertex[x=0,y=0,Math,L={0000},Lpos=270]{0000}
\Vertex[x=10,y=0,Math,L={1100},Lpos=270]{1100}
\Vertex[x=20,y=0,Math,L={1010},Lpos=270]{1010}
\Vertex[x=30,y=0,Math,L={0110},Lpos=270]{0110}
\Vertex[x=0,y=10,Math,L={1000},Lpos=90]{1000}
\Vertex[x=10,y=10,Math,L={0100},Lpos=90]{0100}
\Vertex[x=20,y=10,Math,L={0010},Lpos=90]{0010}
\Vertex[x=30,y=10,Math,L={1110},Lpos=90]{1110}
\AddVertexColor{black}{1000,0100,0010,1110}
\Edge[color=green,style=Dash](0000)(1110)
\Edge[color=green](1000)(0110)
\Edge[color=green,style=Dash](0100)(1010)
\Edge[color=green](1100)(0010)

\Edge(0010)(1010)
\Edge(0110)(1110)
\Edge[color=red](0010)(0110)
\Edge[color=red,style=Dash](1010)(1110)
\Edge[color=blue](0000)(0010)
\Edge[color=blue,style=Dash](1000)(1010)
\Edge[color=blue,style=Dash](0100)(0110)
\Edge[color=blue](1100)(1110)
\Edge(0000)(1000)
\Edge(0100)(1100)
\Edge[color=red](0000)(0100)
\Edge[color=red,style=Dash](1000)(1100)
\end{tikzpicture}
\end{center}

\end{example}

\begin{example}
\label{ex:Lk4}
We can find the signed adjacency matrix and the signed Laplacian of the previous example, as long as we order the vertices.

We number the vertices with the bosons first, left to right, and then the fermions, from left to right.  The signed adjacency matrix is then:
\[A=\left[\begin{array}{cccccccc}
0&0&0&0&1&1&1&-1\\
0&0&0&0&-1&1&1&1\\
0&0&0&0&-1&-1&1&-1\\
0&0&0&0&1&-1&1&1\\
1&-1&-1&1&0&0&0&0\\
1&1&-1&-1&0&0&0&0\\
1&1&1&1&0&0&0&0\\
-1&1&-1&1&0&0&0&0
\end{array}\right]\]
and the signed Laplacian is
\[L=4I-A=\left[\begin{array}{cccccccc}
4&0&0&0&-1&-1&-1&1\\
0&4&0&0&1&-1&-1&-1\\
0&0&4&0&1&1&-1&1\\
0&0&0&4&-1&1&-1&-1\\
-1&1&1&-1&4&0&0&0\\
-1&-1&-1&-1&0&4&0&0\\
-1&-1&-1&-1&0&0&4&0\\
1&-1&1&-1&0&0&0&4
\end{array}\right]\]

\end{example}

\begin{example}
By constructing a prism on the previous example, we get an Adinkra with $N=5$ colors.  This can be viewed as a quotient of a $5$-cube with the code $\{00000,11110\}$.  

More generally, given an Adinkra which is a quotient of an $N$-cube by a code $C$ of length $N$, the prism of this Adinkra is a quotient of an $N+1$-cube by the code of length $N+1$ obtained from $C$ by adding a zero to all codewords in $C$.

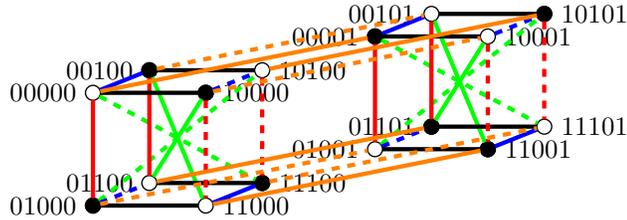
\begin{center}
\begin{tikzpicture}[scale=0.15]
\GraphInit[vstyle=Welsh]
\SetVertexNormal[MinSize=5pt]
\SetUpEdge[labelstyle={draw},style={ultra thick}]
\tikzset{Dash/.style={dashed,draw,ultra thick}}
\Vertex[x=0,y=0,Math,L={00000},Lpos=180]{00000}
\Vertex[x=10,y=0,Math,L={10000},Lpos=0]{10000}
\Vertex[x=0,y=-10,Math,L={01000},Lpos=180]{01000}
\Vertex[x=10,y=-10,Math,L={11000},Lpos=0]{11000}
\Vertex[x=5,y=2,Math,L={00100},Lpos=180]{00100}
\Vertex[x=15,y=2,Math,L={10100},Lpos=0]{10100}
\Vertex[x=5,y=-8,Math,L={01100},Lpos=180]{01100}
\Vertex[x=15,y=-8,Math,L={11100},Lpos=0]{11100}

\Vertex[x=25,y=5,Math,L={00001},Lpos=180]{00001}
\Vertex[x=35,y=5,Math,L={10001},Lpos=0]{10001}
\Vertex[x=25,y=-5,Math,L={01001},Lpos=180]{01001}
\Vertex[x=35,y=-5,Math,L={11001},Lpos=0]{11001}
\Vertex[x=30,y=7,Math,L={00101},Lpos=180]{00101}
\Vertex[x=40,y=7,Math,L={10101},Lpos=0]{10101}
\Vertex[x=30,y=-3,Math,L={01101},Lpos=180]{01101}
\Vertex[x=40,y=-3,Math,L={11101},Lpos=0]{11101}
\AddVertexColor{black}{10000,01000,00100,11100,00001,11001,10101,01101}
\Edge[color=green,style=Dash](00000)(11100)
\Edge[color=green](10000)(01100)
\Edge[color=green,style=Dash](01000)(10100)
\Edge[color=green](11000)(00100)
\Edge(00100)(10100)
\Edge(01100)(11100)
\Edge[color=red](00100)(01100)
\Edge[color=red,style=Dash](10100)(11100)
\Edge[color=blue](00000)(00100)
\Edge[color=blue,style=Dash](10000)(10100)
\Edge[color=blue,style=Dash](01000)(01100)
\Edge[color=blue](11000)(11100)
\Edge(00000)(10000)
\Edge(01000)(11000)
\Edge[color=red](00000)(01000)
\Edge[color=red,style=Dash](10000)(11000)

\Edge[color=green,style=Dash](00001)(11101)
\Edge[color=green](10001)(01101)
\Edge[color=green,style=Dash](01001)(10101)
\Edge[color=green](11001)(00101)
\Edge(00101)(10101)
\Edge(01101)(11101)
\Edge[color=red](00101)(01101)
\Edge[color=red,style=Dash](10101)(11101)
\Edge[color=blue](00001)(00101)
\Edge[color=blue,style=Dash](10001)(10101)
\Edge[color=blue,style=Dash](01001)(01101)
\Edge[color=blue](11001)(11101)
\Edge(00001)(10001)
\Edge(01001)(11001)
\Edge[color=red](00001)(01001)
\Edge[color=red,style=Dash](10001)(11001)
\Edge[color=orange](00000)(00001)
\Edge[color=orange](11000)(11001)
\Edge[color=orange](10100)(10101)
\Edge[color=orange](01100)(01101)
\Edge[color=orange,style=Dash](10000)(10001)
\Edge[color=orange,style=Dash](01000)(01001)
\Edge[color=orange,style=Dash](00100)(00101)
\Edge[color=orange,style=Dash](11100)(11101)
\end{tikzpicture}
\end{center}

We can draw this prism on the $K(4,4)$ graph:
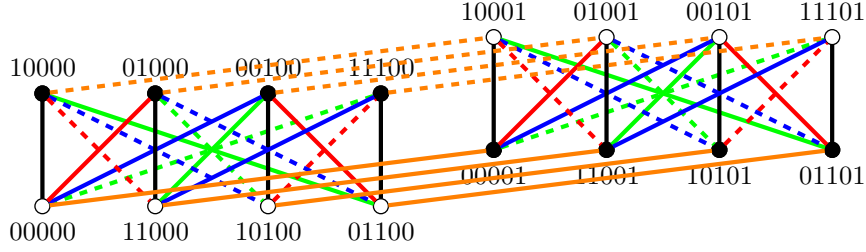
\begin{center}
\begin{tikzpicture}[scale=0.15]
\GraphInit[vstyle=Welsh]
\SetVertexNormal[MinSize=5pt]
\SetUpEdge[labelstyle={draw},style={ultra thick}]
\tikzset{Dash/.style={dashed,draw,ultra thick}}
\Vertex[x=0,y=0,Math,L={00000},Lpos=270]{00000}
\Vertex[x=10,y=0,Math,L={11000},Lpos=270]{11000}
\Vertex[x=20,y=0,Math,L={10100},Lpos=270]{10100}
\Vertex[x=30,y=0,Math,L={01100},Lpos=270]{01100}
\Vertex[x=0,y=10,Math,L={10000},Lpos=90]{10000}
\Vertex[x=10,y=10,Math,L={01000},Lpos=90]{01000}
\Vertex[x=20,y=10,Math,L={00100},Lpos=90]{00100}
\Vertex[x=30,y=10,Math,L={11100},Lpos=90]{11100}

\Vertex[x=40,y=5,Math,L={00001},Lpos=270]{00001}
\Vertex[x=50,y=5,Math,L={11001},Lpos=270]{11001}
\Vertex[x=60,y=5,Math,L={10101},Lpos=270]{10101}
\Vertex[x=70,y=5,Math,L={01101},Lpos=270]{01101}
\Vertex[x=40,y=15,Math,L={10001},Lpos=90]{10001}
\Vertex[x=50,y=15,Math,L={01001},Lpos=90]{01001}
\Vertex[x=60,y=15,Math,L={00101},Lpos=90]{00101}
\Vertex[x=70,y=15,Math,L={11101},Lpos=90]{11101}
\AddVertexColor{black}{10000,01000,00100,11100,00001,11001,10101,01101}
\Edge[color=green,style=Dash](00000)(11100)
\Edge[color=green](10000)(01100)
\Edge[color=green,style=Dash](01000)(10100)
\Edge[color=green](11000)(00100)
\Edge(00100)(10100)
\Edge(01100)(11100)
\Edge[color=red](00100)(01100)
\Edge[color=red,style=Dash](10100)(11100)
\Edge[color=blue](00000)(00100)
\Edge[color=blue,style=Dash](10000)(10100)
\Edge[color=blue,style=Dash](01000)(01100)
\Edge[color=blue](11000)(11100)
\Edge(00000)(10000)
\Edge(01000)(11000)
\Edge[color=red](00000)(01000)
\Edge[color=red,style=Dash](10000)(11000)

\Edge[color=green,style=Dash](00001)(11101)
\Edge[color=green](10001)(01101)
\Edge[color=green,style=Dash](01001)(10101)
\Edge[color=green](11001)(00101)
\Edge(00101)(10101)
\Edge(01101)(11101)
\Edge[color=red](00101)(01101)
\Edge[color=red,style=Dash](10101)(11101)
\Edge[color=blue](00001)(00101)
\Edge[color=blue,style=Dash](10001)(10101)
\Edge[color=blue,style=Dash](01001)(01101)
\Edge[color=blue](11001)(11101)
\Edge(00001)(10001)
\Edge(01001)(11001)
\Edge[color=red](00001)(01001)
\Edge[color=red,style=Dash](10001)(11001)
\Edge[color=orange](00000)(00001)
\Edge[color=orange](11000)(11001)
\Edge[color=orange](10100)(10101)
\Edge[color=orange](01100)(01101)
\Edge[color=orange,style=Dash](10000)(10001)
\Edge[color=orange,style=Dash](01000)(01001)
\Edge[color=orange,style=Dash](00100)(00101)
\Edge[color=orange,style=Dash](11100)(11101)
\end{tikzpicture}
\end{center}

\end{example}

\begin{example}
This is the quotient of an $8$-cube by the $e_8$ code.  Note that this is also a $K(8,8)$ graph.

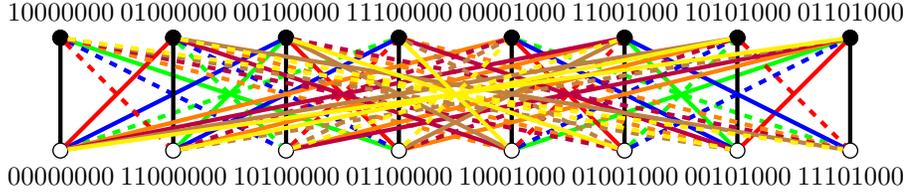
\begin{center}
\begin{tikzpicture}[scale=0.15]
\GraphInit[vstyle=Welsh]
\SetVertexNormal[MinSize=5pt]
\SetUpEdge[labelstyle={draw},style={ultra thick}]
\tikzset{Dash/.style={dashed,draw,ultra thick}}
\Vertex[x=0,y=0,Math,L={00000000},Lpos=270]{00000}
\Vertex[x=10,y=0,Math,L={11000000},Lpos=270]{11000}
\Vertex[x=20,y=0,Math,L={10100000},Lpos=270]{10100}
\Vertex[x=30,y=0,Math,L={01100000},Lpos=270]{01100}
\Vertex[x=0,y=10,Math,L={10000000},Lpos=90]{10000}
\Vertex[x=10,y=10,Math,L={01000000},Lpos=90]{01000}
\Vertex[x=20,y=10,Math,L={00100000},Lpos=90]{00100}
\Vertex[x=30,y=10,Math,L={11100000},Lpos=90]{11100}

\Vertex[x=40,y=10,Math,L={00001000},Lpos=90]{00001}
\Vertex[x=50,y=10,Math,L={11001000},Lpos=90]{11001}
\Vertex[x=60,y=10,Math,L={10101000},Lpos=90]{10101}
\Vertex[x=70,y=10,Math,L={01101000},Lpos=90]{01101}
\Vertex[x=40,y=0,Math,L={10001000},Lpos=270]{10001}
\Vertex[x=50,y=0,Math,L={01001000},Lpos=270]{01001}
\Vertex[x=60,y=0,Math,L={00101000},Lpos=270]{00101}
\Vertex[x=70,y=0,Math,L={11101000},Lpos=270]{11101}
\AddVertexColor{black}{10000,01000,00100,11100,00001,11001,10101,01101}
\Edge[color=green,style=Dash](00000)(11100)
\Edge[color=green](10000)(01100)
\Edge[color=green,style=Dash](01000)(10100)
\Edge[color=green](11000)(00100)
\Edge(00100)(10100)
\Edge(01100)(11100)
\Edge[color=red](00100)(01100)
\Edge[color=red,style=Dash](10100)(11100)
\Edge[color=blue](00000)(00100)
\Edge[color=blue,style=Dash](10000)(10100)
\Edge[color=blue,style=Dash](01000)(01100)
\Edge[color=blue](11000)(11100)
\Edge(00000)(10000)
\Edge(01000)(11000)
\Edge[color=red](00000)(01000)
\Edge[color=red,style=Dash](10000)(11000)

\Edge[color=green,style=Dash](00001)(11101)
\Edge[color=green](10001)(01101)
\Edge[color=green,style=Dash](01001)(10101)
\Edge[color=green](11001)(00101)
\Edge(00101)(10101)
\Edge(01101)(11101)
\Edge[color=red](00101)(01101)
\Edge[color=red,style=Dash](10101)(11101)
\Edge[color=blue](00001)(00101)
\Edge[color=blue,style=Dash](10001)(10101)
\Edge[color=blue,style=Dash](01001)(01101)
\Edge[color=blue](11001)(11101)
\Edge(00001)(10001)
\Edge(01001)(11001)
\Edge[color=red](00001)(01001)
\Edge[color=red,style=Dash](10001)(11001)
\Edge[color=orange](00000)(00001)
\Edge[color=orange](11000)(11001)
\Edge[color=orange](10100)(10101)
\Edge[color=orange](01100)(01101)
\Edge[color=orange,style=Dash](10000)(10001)
\Edge[color=orange,style=Dash](01000)(01001)
\Edge[color=orange,style=Dash](00100)(00101)
\Edge[color=orange,style=Dash](11100)(11101)
\Edge[color=purple](00000)(11001)
\Edge[color=purple,style=Dash](10000)(01001)
\Edge[color=purple](01000)(10001)
\Edge[color=purple,style=Dash](11000)(00001)
\Edge[color=purple](10100)(01101)
\Edge[color=purple,style=Dash](00100)(11101)
\Edge[color=purple,style=Dash](01100)(10101)
\Edge[color=purple](11100)(00101)
\Edge[color=brown](00000)(10101)
\Edge[color=brown,style=Dash](10000)(00101)
\Edge[color=brown,style=Dash](11000)(01101)
\Edge[color=brown](01000)(11101)
\Edge[color=brown,style=Dash](10100)(00001)
\Edge[color=brown](00100)(10001)
\Edge[color=brown](01100)(11001)
\Edge[color=brown,style=Dash](11100)(01001)
\Edge[color=yellow](00000)(01101)
\Edge[color=yellow,style=Dash](10000)(11101)
\Edge[color=yellow](11000)(10101)
\Edge[color=yellow,style=Dash](01000)(00101)
\Edge[color=yellow,style=Dash](10100)(11001)
\Edge[color=yellow](00100)(01001)
\Edge[color=yellow,style=Dash](01100)(00001)
\Edge[color=yellow](11100)(10001)
\end{tikzpicture}
\end{center}

The situations when we have $K(m,m)$ are fairly rare: this only happens when $m=N=1, 2, 4, 8$.  Simply on dimensional grounds these cannot occur for other $N$.  The curious reader may find it interesting to note there is a relation between these cases and the real, complex, quaternionic, and octonionic algebras, see \cite{Baez}.
\end{example}

\section{Eigenvalues of Adinkras}
\label{sec:eigadinkra}
One of the main themes of this paper is that the signed adjacency and signed Laplacian matrix for an Adinkra, even though they are $\#V\times\#V$ matrices, have only two distinct (non-zero) eigenvalues.
Analogous to the theory of {\em strongly regular graphs}, the spectral properties of Adinkras are closely related to the fact that Adinkras are strongly regular signed graphs.

In this section, we first verify these properties and state some corollaries, before surveying and describing the connections between Adinkras and the notions of strongly regular signed graphs in the literature.

We start with a signed analogue of the equivalence between strongly regular graphs and regular graphs having two distinct non-zero eigenvalues.  Here a walk is {\em positive} if it contains an even number of negative edges, and is {\em negative} otherwise.
\begin{proposition} {\cite[Theorem 2.2]{GhasFath}} \label{prop:SRBSG}
Suppose a signed graph is $k$-regular and triangle-free.  Then the adjacency matrix $A$ has two distinct eigenvalues if and only if for every pair of non-adjacent vertices, the number of positive and negative walks of length $2$ between them is equal, in which case $A^2=kI$.

\end{proposition}

We show that Adinkras satisfy the conditions of this Proposition.

\begin{lemma}
Let $v$ and $w$ be distinct vertices of an Adinkra.
The number of positive walks of length $2$ joining $v$ and $w$ is equal to the number of negative walks of length $2$ joining them.
\label{lem:evenodd}
\end{lemma}
\begin{proof}
A walk from $v$ of length $2$ consists of edges $e_1, e_2$.  Let their colors be $c_1, c_2$, respectively.  By the bicolor cycle condition in the definition of Adinkras, there must be another walk from $w$ to $v$ with edges $e_3, e_4$ with colors $c_1, c_2$, respectively.  We can reverse that walk to a walk from $v$ to $w$ with edges $e_4, e_3$, with colors $c_2, c_1$.  By the totally odd condition, this must have opposite sign from the original walk.
This provides a pairing of walks of length $2$ from $v$ to $w$, of opposite sign.
\end{proof}

\begin{theorem}
Let $A$ be the signed adjacency matrix of an Adinkra with $N$ colors. The eigenvalues of  $A$ are $\pm\sqrt{N}$, each of multiplicity $\frac{\#V}{2}$, and
\begin{equation}
A^2=NI.\label{eqn:a2}
\end{equation}
Furthermore,
\[\det A = \begin{cases}
N^{\#V/2},&N> 1\\
(-1)^{\#V/2},&N=1.
\end{cases}\]
\label{thm:eigA}
\end{theorem}
\begin{proof}
In an Adinkra, every vertex is incident to one edge for each of the $N$ colors, so Adinkras are $N$-regular graphs.  Adinkras are bipartite, and so have no triangles.
Then Lemma~\ref{lem:evenodd} provides the remaining hypothesis for Proposition \ref{prop:SRBSG}, hence  $A^2=NI$, and the eigenvalues of $A$ are the roots of $x^2=N$, namely, $\pm\sqrt{N}$.
Since $A$ has zero diagonal terms, the trace is $0$.  As a consequence, the two eigenvalues have equal multiplicity, which is $\#V/2$.

Multiplying all eigenvalues together, we get the determinant of $A$ to be
\[\det A = (\sqrt{N})^{\#V/2}(-\sqrt{N})^{\#V/2}=(-1)^{\#V/2}N^{\#V/2}.\]
If $N\ge 2$, then according to Corollary~\ref{cor:vmult4}, $\#V$ is a multiple of 4, and this formula becomes $N^{\#V/2}$. 
\end{proof}

\begin{corollary}
Let $L$ be the signed Laplacian of an Adinkra with $N$ colors. 
The eigenvalues of $L$ are $N\pm\sqrt{N}$, each of multiplicity $\frac{\#V}{2}$, and
\begin{equation}
L^2=2NL-N(N-1)I.
\label{eqn:l2}
\end{equation}
The determinant of the signed Laplacian of an Adinkra with $N$ colors is
\[\det(L)=(N(N-1))^{\#V/2}.\]
\label{cor:eigL}
\end{corollary}
\begin{proof}
Since each vertex of an Adinkras has degree $N$, we have
\begin{equation}
    L=NI-A.
    \label{eqn:lnia}
\end{equation}
From this, if $\lambda$ is an eigenvalue for $A$, then $N-\lambda$ is an eigenvalue for $L$, of the same multiplicity.

Next, we can write (\ref{eqn:lnia}) as $A=NI-L$, and substituting into (\ref{eqn:a2}) gives us:
\[N^2I-2NL+L^2=NI,\]
which can be rewritten as (\ref{eqn:l2}).

Finally, by multiplying the eigenvalues with the corresponding multiplicities, we get that the determinant is
\[\det(L)=(N+\sqrt{N})^{\#V/2}(N-\sqrt{N})^{\#V/2}=(N(N-1))^{\#V/2}.\]
\end{proof}

If we choose an ordering of the vertices with all of the bosons first, followed by the fermions, then the signed adjacency matrix $A$ takes the following block form:
\begin{equation}\label{eqa}
A=\left[\begin{array}{c|c}
0&X\\\hline
X^T&0
\end{array}\right].
\end{equation}
Each block is a $\#V/2 \times \#V/2$ matrix.  The matrix $X$ in the upper right records the adjacency relations between bosons and fermions: its $(i,j)$ entry is $+1$ if there is a solid edge from boson $i$ to fermion $j$, $-1$ if there is a dashed edge from boson $i$ to fermion $j$, and is $0$ if there is no edge between those vertices.  In the lower left we have $X^T$ because $A$ is symmetric.

Likewise, the signed Laplacian $L$ is:
\begin{equation}\label{eq1}
L=\left[\begin{array}{c|c}
NI&-X\\\hline
-X^T&NI
\end{array}\right].
\end{equation}
Then, as \cite{Ramezani}*{Lemma~2} points out, identity (\ref{eqn:a2}) implies:
\begin{theorem}
\begin{equation}
    XX^T=X^TX=NI.
    \label{eqn:xx}
\end{equation}
\end{theorem}

\begin{corollary}\label{cor:detX}
\[\det(X)=\pm N^{\#V/4}.\]
\end{corollary}
\begin{proof}
We take the determinant of (\ref{eqn:xx}) to see that
\[\det(X)\det(X^T)=N^{\#V/2}.\]
Since $\det(X)=\det(X^T)$,
\[(\det(X))^2=N^{\#V/2}\]
and the result follows.
\end{proof}

The sign of $\det(X)$ is not determined by $N$: for each $N$, there are examples with positive and negative determinant.  For instance, a single vertex switch on a boson (see Section~\ref{sec:dependsig}) will result in that row operation on $X$ that multiplies a row by $-1$, which results in reversing the sign of the determinant.

The first notion of strong regularity in the theory of signed graphs was due to Zaslavsky.
For two vertices $u,v$ in a signed graph, denote by $w_2(u,v)$ the number of positive walks of length 2 between $u,v$ minus the number of negative walks of length 2 between $u,v$.

\begin{definition} \cite{zaslavsky2}
A {\em very strongly regular signed graph} with parameters $t,k,p,\rho_0$ is a $k$-regular signed graph in which the value of $w_2(u,v)$ for any pair of adjacent (resp. non-adjacent) vertices $u,v$ is $t\mu(uv)$ (resp. $p$).
Moreover, the signed graph is $\rho_0$-signed-regular, i.e., for every vertex $v$, the number of positive edges incident to $v$ minus the number of negative edges incident to $v$ is $\rho_0$.
\end{definition}

While an Adinkra does not satisfy the signed-regular condition in general, it satisfies all other conditions with $p=t=0$ as we have seen from the above. 

Ramezani gave a more general definition of {\em strongly regular signed graphs} in \cite{Ramezani2}, where the signed-regular condition in Zaslavsky's definition is removed, and $w_2(u,v)$ is allowed to be either $p$ or $-p$ for non-adjacent vertices.
Stani\'c gave another definition of strongly regular signed graphs that generalizes Zaslavsky's definition in \cite{Stanic3}, where the signed-regular condition is again removed, and instead of using one parameter $t$ for adjacent vertices, one has two parameters $a,b$ so that $w_2(u,v)=a$ (resp. $w_2(u,v)=b$) if $u,v$ are connected by a positive (resp. negative) edge.
It is evident that Adinkras belong to both notions, with $p=t=0$ and $p=a=b=0$, respectively.

\section{Critical Groups of Adinkras}
\label{sec:criticalgroup}

Here we define the critical group of a signed graph in terms of the signed Laplacian and investigate its structure for Adinkras.  
In addition, we develop algebraic machinery to study critical groups (over $
\zz$) by considering cokernels of matrices over other rings, which is applicable beyond the case of Adinkras.  

\begin{definition}
The critical group of an Adinkra is the cokernel of the signed Laplacian for the Adinkra as a signed graph. 
\end{definition}

In the unsigned case, the graph Laplacian is singular, with the multiplicity of the zero eigenvalue equal to the number of connected components of the graph.  In studying the critical group, one often works with a reduced Laplacian or specifies that we are interested in the torsion part of the cokernel.  For Adinkras, the Laplacian will always be non-singular, except in the degenerate cases of $N=0$ and $N=1$.

We next study the form of the critical group of an Adinkra.   A number of partial results and cases are presented, along with a table of small examples---see Section~\ref{sec:table}.  
We start with the following general fact about invariant factors, see e.g.  \cite{StanleyNormal} or \cite{ChipBook}*{Chapter 4}. 
\begin{theorem} [Elementary Divisor Theorem]
Let $M$ be a matrix over a unique factorization domain such that $M$ has a Smith Normal Form ${\rm diag}(d_1,\ldots,d_n)$.
Then $d_1\ldots d_k=g_k$ (up to units), where $g_k$ is the gcd of all $k\times k$ minors of $M$.
\label{thm:edt}
\end{theorem}

\subsection{Table of invariant factors: Small Examples}
\label{sec:table}
The invariant factors of various $(N,k)$-Adinkras are collected in the table below for $N \leq 8, k \leq 4$.
For each entry, the doubly even code
defining the Adinkra is given.  The invariant factors are written as
\[({y_1}^{\alpha_1},\ldots,{y_k}^{\alpha_k})\]
where $y_1\,|\,y_2\,|\,\cdots\,|\,y_k$, and the exponents indicate multiplicity.  For instance, $(1^4,6^4)$ means the invariant factors are $(1,1,1,1,6,6,6,6)$.   The number of invariant factors of $L$, counting multiplicity, is $\#V=2^{N-k}$.  In the table:\\

The doubly even code is written on the left.

The code $t^j$ means the trivial code of length $j$ consisting of $\{0\cdots 0\}$.

The codes $d_n$, $e_7$, and $e_8$, as well as direct sums of codes, are defined in Section~\ref{sec:quotient}. 

The code $h_8$ has length $8$ and has one generator:

\[\left[
\begin{array}{cccccccc}
1&1&1&1&1&1&1&1
\end{array}\right]
\]

\begin{center}
\resizebox{\textwidth}{!}{
    \begin{tabular}{c|l|l|l|l|l|}
    $N$&$k=0$&$k=1$&$k=2$&$k=3$&$k=4$\\\hline
    1&$t$: $(1,0)$\\\hline
    2&$t^2$: $(1^2,2^2)$\\\hline
    3&$t^3$: $(1^4,6^4)$\\\hline
    4&$t^4$: $(1^8,12^8)$&$d_4$: $(1^2,2^2,6^2,12^2)$\\\hline
    5&$t^5$: $(1^{16},20^{16})$&$d_4\oplus t$: $(1^8,20^8)$\\\hline
    6&$t^6$: $(1^{32},30^{32})$&$d_4\oplus t^2$: $(1^{16},30^{16})$&$d_6$: $(1^8,30^8)$\\\hline
    7&$t^7$: $(1^{64},42^{64})$& $d_4\oplus t^3$: $(1^{32},42^{32})$&
    $d_6\oplus t$: $(1^{16},42^{16})$&
    $e_7$: $(1^{8},42^{8})$\\\hline
    8&$t^8$: $(1^{128},56^{128})$ & $d_4\oplus t^4$: $(1^{64},56^{64})$& $d_6\oplus t^2$: $(1^{32},56^{32})$ & $e_7\oplus t$: $(1^{16},56^{16})$\\
    &&$h_8$: $(1^{56},2^8,28^8,56^{56})$ & $d_4\oplus d_4$: $(1^{24},2^8,28^8,56^{24})$ & $d_8$: $(1^8,2^8,28^8,56^8)$ & $e_8$: $(1^2,2^6,28^6,56^2)$
    \end{tabular}}
\end{center}

Based on these cases, we  observe:

\begin{itemize}
    \item The first invariant factor is always $1$, and the last invariant factor is always $N(N-1)$.
    \item Each invariant factor occurs with even multiplicity.
    \item The product of the $i$th invariant factor and the $(\#V-i+1)$th invariant factor is $N(N-1)$.
\end{itemize}
The second and third of these observations will follow from Theorem~\ref{thm:inv_L}, and the first observation will be proven in Theorem~\ref{thm:firstlast}.

\subsection{General structure of the Invariant Factors} \label{sec:structure_invf}

The invariant factors for the signed Laplacian of an Adinkra turn out to relate to the invariant factors of the matrix $X$ that encodes the signed adjacency relations from fermions to bosons, as described in Section \ref{sec:eigadinkra}.  We begin this section by proving a few facts about the invariant factors for $X$.  We then state and prove a formula for the invariant factors for the signed Laplacian of an Adinkra in terms of the invariant factors of $X$.  Finally, we explore the invariant factors for $X$ for a number of cases.

\begin{proposition} \label{prop:symm_invf}
The invariant factors of $X$ all divide $N$, and if the invariant factors are
\[(x_1,\ldots,x_{\#V/2})\]
then
\[(x_1,\ldots,x_{\#V/2})=(N/x_{\#V/2},\ldots,N/x_1).\]
\end{proposition}

In other words, the invariant factors for $X$ are determined by the first $\#V/4$ invariant factors: the others are obtained by dividing these from $N$.

\begin{proof}
Let $D_1$ be the Smith Normal Form for $X$, i.e.,
\begin{equation}
    D_1=BXC,
    \label{eqn:bxc}
\end{equation}
where $B$ and $C$ are $\#V/2 \times \#V/2$ integer matrices, with integer inverses, and $D_1$ is diagonal, and where each entry on the diagonal divides the next.

By Corollary~\ref{cor:detX}, $\det(X)\not=0$, so $X$ has an inverse with rational entries.  The inverse of $D_1$ is
\[D_1^{-1}=C^{-1}X^{-1}B^{-1}.\]
We use $XX^T=NI=X^TX$, so that $X^T=NX^{-1}$, to get:
\begin{equation}
    C^{-1}X^TB^{-1}=ND_1^{-1}.
\label{eqn:cyb}
\end{equation}
Since $D_1$ is diagonal, so is $D_1^{-1}$, and since the left side is an integer matrix, so is $ND_1^{-1}$.  This proves that the invariant factors for $X$ all divide $N$.

Let $D_2=ND_1^{-1}$.  Since $D_1$ is diagonal, with diagonal entries $(x_1,\ldots,x_{\#V/2})$, we have that $D_2$ is diagonal, with diagonal entries $(N/x_1,\ldots,N/x_{\#V/2})$.  Each entry is a multiple of the next one.  On the other hand, we have (\ref{eqn:cyb}), which implies that $D_2$ is the Smith Normal form for $X^T$, except that the diagonal entries are in reverse order.  So the invariant factors of $X^T$ are
\[(N/x_{\#V/2},\ldots,N/x_1).\]
Now the Smith Normal Form for $X^T$ is the same as that for $X$.  Thus, 
\[(x_1,\ldots,x_{\#V/2})=(N/x_{\#V/2},\ldots,N/x_1).\]
\end{proof}

We are now ready to state the relationship between the invariant factors for $X$ and the invariant factors for the signed Laplacian.
\begin{theorem} \label{thm:inv_L}
If
\[(x_1,\ldots,x_{\#V/4},N/x_{\#V/4},\ldots,N/x_1)\]
are the invariant factors for $X$, then the invariant factors for the signed Laplacian are
\[(x_1,x_1,x_2,x_2,\ldots,x_{\#V/4},x_{\#V/4} ,N(N-1)/x_{\#V/4},N(N-1)/x_{\#V/4},\ldots,N(N-1)/x_1,N(N-1)/x_1).\]
In other words, we take the invariant factors for $X$, multiply each of the largest $\#V/4$ factors by $N-1$, then double the multiplicity of all invariant factors.
\end{theorem}

\begin{proof}
As in the proof of the previous theorem, let $B$ and $C$ be integer matrices with integer inverses, and let
\[D_1=BXC\]
be the Smith Normal Form for $X$ and let $D_2=ND_1^{-1}=C^{-1}X^TB^{-1}$.  Define
\begin{align*}
    E&=\begin{bmatrix}
    B&0\\
    0&C^{-1}
    \end{bmatrix}
    \\
    F&=\begin{bmatrix}
    B^{-1}&0\\
    0&C
    \end{bmatrix}.
\end{align*}
If $L$ is the signed Laplacian, then
\begin{align*}
ELF&=
\begin{bmatrix}
    B&0\\
    0&C^{-1}
\end{bmatrix}
\begin{bmatrix}
    NI&-X\\
    -X^T&NI
\end{bmatrix}
\begin{bmatrix}
    B^{-1}&0\\
    0&C
\end{bmatrix}\\
&=
\begin{bmatrix}
    NI&-BXC\\
    -C^{-1}X^TB^{-1}&NI
\end{bmatrix}\\
&=
\begin{bmatrix}
    NI&-D_1\\
    -D_2&NI
\end{bmatrix}.
\end{align*}
The rows and columns of $ELF$ can be rearranged so that it is a block diagonal matrix with $2\times 2$ blocks.  If $x_i$ is the $i$th diagonal entry for $D_1$, and $y_i$ is the $i$th diagonal entry for $D_2$, then the corresponding $2\times 2$ block is
\[\begin{bmatrix}
    N&-x_i\\
    -y_i&N
\end{bmatrix}.
\]
As described above, the invariant factors for $X$ are
\[(x_1,\ldots,x_{\#V/2})=(y_{\#V/2},\ldots,y_1).\]
If $i\le \#V/4$, then $x_i$ divides $y_i$; otherwise $y_i$ divides $x_i$.  Without loss of generality $i\le \#V/4$; the following comments apply when $i>\#V/4$ but with $x_i$ and $y_i$ reversed, and with the transpose of the above matrix.

Then $x_i$ divides $y_i=N/x_i$.  Swapping the two sets of columns gives
\[\begin{bmatrix}
    -x_i&N\\
    N&-N/x_i
\end{bmatrix}.
\]
Adding a multiple of the first row to the second gives
\[\begin{bmatrix}
    -x_i&N\\
    0&-N/x_i+N^2/x_i
\end{bmatrix}.
\]
Adding a multiple of the first column to the second, then multiplying the first row by $-1$, results in the diagonal matrix:
\[\begin{bmatrix}
    x_i&0\\
    0&N(N-1)/x_i
\end{bmatrix}.
\]
Thus, this $2\times 2$ block gives invariant factors $x_i$ and $N(N-1)/x_i$.  Together, for all $x_i$ with $i\le \#V/4$, this procedure provides the invariant factors
\begin{equation}
(x_1,\ldots,x_{\#V/4},N(N-1)/x_{\#V/4},\ldots,N(N-1)/x_1).
\label{eqn:firstinv}
\end{equation}
The analogous argument for $i>\#V/4$ involves $y_i$ replacing $x_i$, but this sequence of numbers is
\[(y_{\#V/4+1},\ldots,y_{\#V/2})=(x_{\#V/4},\ldots,x_1),\]
which provides the same invariant factors as (\ref{eqn:firstinv}) again.  Thus, we have double the multiplicity of the invariant factors in (\ref{eqn:firstinv}).

Thus, if the invariant factors for $X$ are
\[(x_1,x_2,\ldots,x_{\#V/4},N/x_{\#V/4},\ldots,N/x_1)\]
then the invariant factors for the signed Laplacian are
\[(x_1,x_1,x_2,x_2,\ldots,x_{\#V/4},x_{\#V/4},N(N-1)/x_{\#V/4},N(N-1)/x_{\#V/4},\ldots,N(N-1)/x_1,N(N-1)/x_1).\]
\end{proof}

\begin{theorem}\label{thm:firstlast}
The first invariant factor for $X$ is $1$, and the last invariant factor for $X$ is $N$.

The first two invariant factors for $L$ are $1$, and the last two invariant factors for $L$ are $N(N-1)$.
\end{theorem}
\begin{proof}
The $X$ matrix, being part of a signed adjacency matrix for a simple graph, has only $0$, $1$, and $-1$ entries.  There are, in fact, edges, so there are entries of either $1$ or $-1$.

Then the fact that the first invariant factor for $X$ is $1$ follows from the Elementary Divisor Theorem (Theorem~\ref{thm:edt}).
The second claim follows from Theorem~\ref{thm:inv_L}.
\end{proof}

Based on these theorems, it suffices to look at the invariant factors for $X$ rather than the signed Laplacian.  Furthermore we can look at the first $\#V/4$ factors, because the others are $N$ divided by the first factors. 

\begin{corollary}
The last half of the invariant factors for $L$ are all even.
\label{cor:lasteven}
\end{corollary}
\begin{proof}
There are two cases: either the first half of the invariant factors for $L$ contains an even $x_i$ or they are all odd.

In the first case, since each invariant factor divides the next, the even number $x_i$ divides all of the last half of the invariant factors for $L$ and we are done.

In the second case, Theorem~\ref{thm:inv_L} says that the last half of the invariant factors are all of the form $N(N-1)/x_j$.  Now $N(N-1)$ is even, and since all $x_i$ are odd, $N(N-1)/x_j$ is even for all $j$.
\end{proof}

\subsection{Prisms}

\begin{theorem} \label{thm:prism}
If $A$ is a prism, then the invariant factors for $X$ are
\[(1^{\#V/4},N^{\#V/4})\]
and the invariant factors for the signed Laplacian are
\[(1^{\#V/2},(N(N-1))^{\#V/2}).\]
\end{theorem}

\begin{proof}
If $A$ is a prism of $A'$, then $X$ for $A$ is
\[\begin{bmatrix}
    X'&-I\\
    I&X'{}^T
\end{bmatrix}
\]
where $X'$ is the $X$ matrix for $A'$.

We use Theorem~\ref{thm:edt} on the $I$ submatrix on the lower left corner to get that the first $\#V/4$ invariant factors for $X$ are $1$.

Applying Theorem~\ref{thm:inv_L}, the invariant factors for $X$ are $(1^{\#V/4},N^{\#V/4})$, and the invariant factors for $L$ are $(1^{\#V/2},(N(N-1))^{\#V/2})$.
\end{proof}

\begin{corollary} \label{coro:cubic_Adinkra}
The invariant factors for $X$ for an $N$-cube Adinkra are
\[(1^{\#V/4},N^{\#V/4})\]
and the invariant factors for the signed Laplacian of an $N$-cube are
\[(1^{\#V/2},(N(N-1))^{\#V/2}).\]
\end{corollary}

\begin{proof}
Cubical Adinkras of dimension $\geq 1$ are prisms (over cubical Adinkras of lower dimension).
\end{proof}

\section{Invariant factors: odd primes}
\label{sec:oddprime}

The goal of this section is to prove the following theorem.  
\begin{theorem}[The Odd Prime Theorem]\label{thm:oddprime}
Let $p$ be an odd prime dividing $N$.  Then:
\begin{enumerate}
    \item $p$ does not divide any of the first $\#V/4$ invariant factors of $X$.
    \item $p$ does not divide any of the first $\#V/2$ invariant factors of $L$.
\end{enumerate}
\end{theorem}

First, it will be necessary to enhance the matrices somewhat to keep track of the colors of the edges.

\subsection{Laplacians with indeterminates}
\label{sec:indet}

In an Adinkra, edges come with colors.  As such, we can define $N$ adjacency matrices $A_1,\ldots,A_N$: one for each color. The $j,k$ entry for $A_i$ is:
\[(A_i)_{j,k} = \begin{cases}
\mu(e),&\mbox{if there is an edge $e$ of color $i$ from $v_j$ to $v_k$}\\
0,&\mbox{otherwise.}
\end{cases}\]
Then $A_i$ is an integer-valued symmetric matrix, and
\[A=A_1+\cdots+A_N.\]

Let $e_1,\ldots,e_{\#V}$ be the standard basis of $\re^{\#V}$.  Pick a vertex $v_k$.  By the definition of Adinkras, there is a unique edge of color $i$ incident to $v_k$.  If we let $\mu$ be the sign of that edge and suppose that the edge connects $v_k$ to $v_j$, then
\[A_i(e_k)=\mu e_j.\]

Compare the similarity to the $Q_i$ operators defined in (\ref{eqn:qb}) and (\ref{eqn:qf}).  We do not have $\sqrt{-1}$ or $\frac{d}{dt}$, but otherwise $A_i$ is like $Q_i$.

In this situation, the same edge connects $v_j$ to $v_k$, with the same sign, so
\[A_i(e_j)=\mu e_k.\]
Therefore for all $k$,
\[A_i(A_i(e_k))=e_k\]
and since the $e_k$ are arbitrary and form a basis for $\re^{\#V}$, we have
\begin{proposition}
For all $i$,
\begin{equation}
   A_i{}^2=I. \label{eqn:aialg1}
\end{equation}
For all $i\not= j$,
\begin{equation}
   A_iA_j=-A_jA_i.\label{eqn:aialg2}
\end{equation}
\end{proposition}
Compare these with (\ref{eqn:susyalg1}) and (\ref{eqn:susyalg2}).
\begin{proof}
Let $v_k$ be any vertex.  If we let $v_\ell$ be the unique vertex adjacent to $v_k$ through an edge of color $i$, and $\mu$ is its sign, then
\[A_i(e_k)=\mu e_\ell.\]
Likewise, the same edge connects $v_\ell$ to $v_k$ and so
\[A_i(e_\ell)=\mu e_k.\]
Thus, for all $k$,
\[A_i(A_i(e_k))=e_k\]
and (\ref{eqn:aialg1}) more generally follows by linearity.

Let $v_m$ be the vertex so that there is an edge of color $j$ from $v_\ell$ to $v_m$.

Let $v_n$ be the vertex so that there is an edge of color $j$ from $v_k$ to $v_n$.

So there is a walk from $v_n$ to $v_k$ to $v_\ell$ to $v_m$, with colors $j$, then $i$, then $j$.  Because bicolor cycles are of length 4 in an Adinkra, there must be an edge of color $i$ from $v_n$ to $v_m$.  In other words,
\begin{align*}
A_j(A_i(e_k))&=\mu(e_k,e_\ell)\mu(e_\ell,e_m) e_m\\
A_i(A_j(e_k))&=\mu(e_k,e_n)\mu(e_n,e_m) e_m.
\end{align*}
By the totally odd property of Adinkras, these will have opposite sign, and so
\[A_j(A_i(e_k))=-A_i(A_j(e_k))\]
for all $e_k$.  Then (\ref{eqn:aialg2}) follows by linearity.
\end{proof}

Now we work in the ring $R=\zz[x_1,\ldots,x_n]$, consisting of polynomials in the formal indeterminates $x_1,\ldots,x_N$.  

Define the colored adjacency matrix $\hat{A}_R$ by
\[\hat{A}_R=\sum_{i=1}^N  A_i x_i.\]

Then the entries of $\hat{A}_R$ are:
\[(\hat{A}_R)_{ij}=\begin{cases}
x_c,&\mbox{if there is a solid edge of color $c$ from vertex $i$ to vertex $j$}\\
-x_c,&\mbox{if there is a dashed edge of color $c$ from vertex $i$ to vertex $j$}\\
0,&\mbox{if there is no edge from vertex $i$ to vertex $j$.}
\end{cases}
\]
The ordinary signed adjacency matrix is obtained by assigning $x_1=\cdots=x_N=1$.

As in Theorem~\ref{thm:eigA}, we have the following:
\begin{theorem}
\begin{equation}
    \hat{A}_R{}^2=(x_1{}^2+\cdots+x_N{}^2)I.
\label{eqn:arsquared}
\end{equation}
\label{thm:arsquared}
\end{theorem}
\begin{proof}
We use (\ref{eqn:aialg1}) and (\ref{eqn:aialg2}).
\begin{align*}
    \hat{A}_R\hat{A}_R
    &=\sum_{i=1}^N \sum_{j=1}^N A_iA_j x_ix_j\\
    &=\sum_{i=j}  A_iA_j x_ix_j + \sum_{i<j} A_iA_j x_ix_j+ \sum_{i>j} A_iA_j x_ix_j\\
    &=\sum_{i=1}^N  A_i{}^2 x_i{}^2 + \sum_{i<j} (A_iA_j+A_jA_i) x_ix_j\\
    &=\sum_{i=1}^N x_i{}^2
\end{align*}
\end{proof}
In what follows, we will be using the notation:
\begin{align*}
    \sigma&=x_1+\cdots+x_N\\
    \rho&=x_1{}^2+\cdots+x_N{}^2.
\end{align*}
The ordinary signed adjacency matrix, where $x_1=\cdots=x_N=1$, we would have $\sigma=\rho=N$.

\begin{corollary}
The operator $\hat{A}_R$ has eigenvalues in a field containing $\sqrt{\rho}$, and the eigenvalues of $\hat{A}_R$ are $\pm\sqrt{\rho}$.  These have equal multiplicity.

In addition,
\[\det(\hat{A}_R)=\begin{cases}
\rho^{\#V/4},&N>1\\
(-1)^{\#V/2}x_1{}^{\#V},&N=1.
\end{cases}
\]
\label{cor:hataeig}
\end{corollary}
\begin{proof}
The proof is identical to that of Theorem~\ref{thm:eigA} but using Theorem~\ref{eqn:arsquared}.
\end{proof}
Define the colored Laplacian $\hat{L}_R$ to be
\[\hat{L}_R=\sigma I-\hat{A}_R\]

\begin{corollary}
For an Adinkra with $N$ colors, the operator $\hat{L}_R$ has eigenvalues in a field containing $\sqrt{\rho}$, and the eigenvalues are $\sigma\pm\sqrt{\rho}$.  These have equal multiplicity.  In addition,

\begin{align}
    \det(\hat{L}_R)&=(\sigma^2-\rho)^{\#V/2}\label{eqn:dethat1}\\
&=(\sum_{1\leq i, j, \leq N}x_ix_j)^{\#V/2}.\label{eqn:dethat2}
\end{align}

\label{cor:hatleig}
\end{corollary}
\begin{proof}
If $\lambda$ is an eigenvalue for $\hat{A}_R$, then $\sigma-\lambda$ is an eigenvalue of $\hat{L}_R$, of the same multiplicity.  Thus by Corollary~\ref{cor:hataeig}, the eigenvalues are $\sigma\pm\sqrt{\rho}$.

Multiplying the eigenvalues with the corresponding multiplicities gives
\[\det(\hat{L}_R)=(\sigma+\sqrt{\rho})^{\#V/2}(\sigma-\sqrt{\rho})^{\#V/2}=(\sigma^2-\rho)^{\#V/2}.\]
The second form (\ref{eqn:dethat2}) comes from multiplying out $\sigma^2$ and subtracting $\rho$.
\end{proof}
\begin{theorem}
For an Adinkra with $N$ colors,
\[\hat{L}_R{}^2-2\sigma \hat{L}_R+(\sigma^2-\rho)I=0.\]
\end{theorem}
\begin{proof}
We compute:
\begin{align*}
    \hat{L}_R{}^2&=(\sigma I - \hat{A}_R)^2\\
    &=\sigma^2I-2\sigma \hat{A}_R+\hat{A}_R{}^2\\
    &=(\rho+\sigma^2)I-2\sigma \hat{A}_R\\
    &=(\rho+\sigma^2)I-2\sigma (\sigma I - \hat{L}_R)\\
    &=(\rho-\sigma^2)I+2\sigma \hat{L}_R\\
\end{align*}
\end{proof}

If the order of the vertices has all the bosons first, followed by all of the fermions, then $\hat{A}_R$ has the form:
\begin{equation}
\hat{A}_R=\begin{bmatrix}
0&-\hat{X}_R\\
-\hat{X}_R{}^T&0
    \end{bmatrix}.
\label{eqn:xrmatrix}
\end{equation}

\begin{theorem}\label{thm:chroXdet}
If $\hat{X}_R$ is defined as above, then
\begin{align*}
\hat{X}_R{}^T\hat{X}_R&=\rho I\\
\hat{X}_R\hat{X}_R{}^T&=\rho I
\end{align*}
and
\[\det \hat{X}_R=\pm \rho^{\#V/4}.\]
\end{theorem}    
\begin{proof}
The first two equations are obtained by squaring (\ref{eqn:xrmatrix}) and using Theorem~\ref{thm:arsquared}.  Taking the determinant of either of these first two equations gives the third.
\end{proof}

\subsection{The Odd Primes Theorem} \label{subsec:oddprime}
Take the matrix $\hat{X}_R$ in $\zz[x_1,\ldots,x_N]$ from the previous section, and set $x_1=x,x_2=\ldots=x_N=1$.\footnote{The arbitrariness of the choice of color here is not obvious, because in general the color classes are not symmetric.}  The resulting matrix is defined over $\zz[x]$ and which we denote $\hat{X}$.  The original $X$ matrix is obtained by setting $x=1$.  Likewise we have $\hat{L}$ as the colored signed Laplacian with $x_1=x$ and $x_2=\ldots=x_N=1$.

We can take the entries of $\hat{X}$ modulo a prime $p$ and obtain a matrix $\tilde{X}$ over $\gf_p[x]$.
Note that setting $x=1$ in $\tilde{X}$ results in the matrix $\overline{X}$ obtained from taking the entries of $X$ modulo $p$.  Likewise we take $\hat{L}$ modulo $p$ to get $\tilde{L}$, and set $x=1$ with $\tilde{L}$ to get $\overline{L}$.

The rings and matrices we consider in this section can be summarized as follows:

\begin{equation} \label{eq:CD}
    \begin{CD}
\zz[x]     @>x\mapsto 1>>  \zz\\
@VVV        @VVV\\
\gf_p[x]     @>x\mapsto 1>>  \gf_p
\end{CD}
\ \ ,\ \ 
\begin{CD}
\hat{X}     @>>>  X\\
@VVV        @VVV\\
\tilde{X}     @>>>  \overline{X}
\end{CD}
\ \ ,\ \ 
\begin{CD}
\hat{L}     @>>>  L\\
@VVV        @VVV\\
\tilde{L}     @>>>  \overline{L}
\end{CD}
\end{equation}

We start with a simple observation that computing the  Smith Normal Form commutes with ring homomorphisms.

\begin{lemma} \label{lem:ringSNF}
Let $\phi:R\rightarrow S$ be a ring homomorphism. Let $M$ be a matrix over $R$ whose Smith Normal Form is $M^{SNF}$. Then $\phi(M^{SNF})$ is a Smith Normal Form of $\phi(M)$ over $S$.
\end{lemma}

\begin{proof}
Let $M^{SNF}=BMC$ where $B,C$ are matrices over $R$ whose determinants are units in $R$. Then we have $\phi(M^{SNF})=\phi(B)\phi(M)\phi(C)$.
Now $\det\phi(B),\det\phi(C)$ are units in $S$ and $\phi(M^{SNF})$ is a diagonal matrix satisfying the divisibility condition for diagonal entries, so it is a Smith Normal Form for $\phi(M)$.
\end{proof}

We study the Smith Normal Form of $\tilde{L}$ with respect to a prime, and compare it to the Smith Normal Form of $L$ via $\overline{L}$. The key idea is the following.

\begin{lemma} \label{lem:XvsTildeX}
Let $p$ be a prime. The following are all equal:
\begin{itemize}
    \item The number of invariant factors of $X$ divisible by $p$
    \item The number of invariant factors of $\tilde{X}$ divisible by $x-1$.
    \item The number of invariant factors of $\overline{X}$ that are $0$
\end{itemize}
Likewise, the following are all equal:
\begin{itemize}
    \item The number of invariant factors of $L$ divisible by $p$
    \item The number of invariant factors of $\tilde{L}$ divisible by $x-1$.
    \item The number of invariant factors of $\overline{L}$ that are $0$
\end{itemize}
\end{lemma}

\begin{proof}
Applying the respective morphisms to $\gf_p$, the first two numbers are both equal to the third.
\end{proof}

The quantity in the above lemma is often known as the {\em $p$-corank} of the matrix $X$, as it is $n$ minus (the rank of $X$ modulo $p$).
Equivalently it is the corank of $\overline{X}$.

\begin{theorem}[The Odd Prime Theorem]\label{thm:oddprime2}
Let $p$ be an odd prime dividing $N$.  Then:
\begin{enumerate}
    \item $p$ does not divide any of the first $\#V/4$ invariant factors of $X$.
    \item $p$ does not divide any of the first $\#V/2$ invariant factors of $L$.
\end{enumerate}
\end{theorem}
\begin{proof}
The product of the invariant factors of $\tilde{X}$ is the determinant.  By Theorem~\ref{thm:chroXdet}, the determinant of $\tilde{X}$ is
\[\det(\tilde{X})=\pm (x^2+N-1)^{\#V/4}=\pm (x^2-1)^{\#V/4} = \pm (x-1)^{\#V/4}(x+1)^{\#V/4}.\]
The factors $x+1$ and $x-1$ are not associates in $\gf_p[x]$ if $p\not=2$.  Thus, $x-1$ appears with a multiplicity of $\#V/4$.  Since each invariant factor divides the next, at most $\#V/4$ invariant factors of $\tilde{X}$ are multiples of $x-1$, and these must be among the last $\#V/4$ invariant factors.  By Lemma~\ref{lem:XvsTildeX}, at most $\#V/4$ invariant factors of $X$ are divisible by $p$, and these must be among the last $\#V/4$ invariant factors.

A similar argument can be done for $L$:
Using Corollary~\ref{cor:hatleig}, we know that
\[\det(\hat{L}_R)=(\sigma^2-\rho)^{\#V/2}\]
and setting $x_2=\cdots=x_N=1$, we get:
\begin{align*}
\det(\hat{L})&=((x+N-1)^2-(x^2+N-1))^{\#V/2}\\
&=(2(N-1)x+(N-1)^2-(N-1))^{\#V/2}\\
&=((N-1)(2x+N-2))^{\#V/2}.
\end{align*}
If we reduce this modulo $p$, then we get
\[\det(\tilde{L})=((0-1)(2x+0-2))^{\#V/2}=((-2)(x-1))^{\#V/2}.\]
Thus, $x-1$ appears with multiplicity $\#V/2$.  Since each invariant factor divides the next, at most $\#V/2$ invariant factors of $\tilde{X}$ are multiples of $x-1$, and these must be among the last $\#V/2$ invariant factors.  By Lemma~\ref{lem:XvsTildeX}, at most $\#V/2$ invariant factors of $X$ are divisible by $p$, and these must be among the last $\#V/2$ invariant factors.
\end{proof}
Note that the version of this theorem for $X$ and for $L$ are equivalent because of Theorem~\ref{thm:inv_L}, so only half of the above proof is necessary.  We include both, however, so that in the next section, Section~\ref{sec:altoddprime}, we can discuss relationships with other theorems independently of Theorem~\ref{thm:inv_L}.

\subsection{Invariant factors up to powers of 2}
Together with what we know about the invariant factors of $X$ and $L$, the Odd Prime Theorem determines all of the invariant factors, up to powers of 2.  One way to state this is the following:
\begin{theorem} \label{thm:main_invf}
Suppose $L$ is the signed Laplacian of an Adinkra on $N$ colors, written as
\[L=\begin{bmatrix}
NI&-X\\
-X^T&NI
\end{bmatrix}\]
Then:
\begin{enumerate}
\item The first $\#V/4$ invariant factors of $X$ are powers of $2$ that divide $N$, and the remaining $\#V/4$ invariant factors are $N$ divided by powers of $2$.
\item The first $\#V/2$ invariant factors of $L$ are powers of $2$ that divide $N$, and the remaining $\#V/2$ invariant factors are $N(N-1)$ divided by powers of $2$.
\end{enumerate} 
\end{theorem}

\begin{proof}
By Theorem~\ref{thm:oddprime}, the first $\#V/4$ invariant factors of $X$ have no odd prime factors.  Thus, they must be powers of $2$.

By Proposition~\ref{prop:symm_invf}, the last $\#V/4$ invariant factors of $X$ are $N$ divided by a corresponding invariant factor from among the first $\#V/4$ factors, which are powers of 2.

By Theorem~\ref{thm:inv_L}, the first $\#V/2$ invariant factors of $L$ are powers of 2, and the last $\#V/2$ invariant factors are $N(N-1)$ divided by powers of 2.
\end{proof}

We already know that the invariant factors are $(1^{\#V/2},(N(N-1))^{\#V/2})$ when the Adinkra is a prism.
We can identify more examples from combining the above Theorem with divisibility arguments, as follows:

\begin{corollary} \label{coro:multi4}
If $N$ is not a multiple of $4$, then the invariant factors for $X$ are
\[(1^{\#V/4},N^{\#V/4}),\]
and the invariant factors for $L$ are
\[(1^{\#V/2},(N(N-1))^{\#V/2}).\]
\end{corollary}

\begin{proof}
The first $\#V/4$ invariant factors for $X$ must be powers of $2$ that are factors of $N$.  If $x_j$ is an invariant factor with $j\le \#V/4$, and $x_j=2^m$, then
\[2^m=x_j\,|\,x_{\#V/2-j+1}=N/x_j=N/2^m.\]
Then $2^{2m}$ divides $N$.

If $N$ is not a multiple of 4, then this means $m=0$, so $x_j=1$.
\end{proof}

The above proof more generally gives us:
\begin{corollary}\label{coro:multi4m}
If $N$ is not a multiple of $4^m$, then each of the first $\#V/4$ invariant factors for $X$ is at most $2^{m-1}$.
\end{corollary}

\subsection{Relationship to two distinct eigenvalues}
\label{sec:altoddprime}
We showed in Section~\ref{sec:eigadinkra} that signed Laplacians for Adinkras have two eigenvalues.  We have now been talking about the invariant factors of the Laplacian.  It turns out there is a relationship between these facts.  We will here give an alternative proof of Theorem~\ref{thm:main_invf}, which does not involve Theorem~\ref{thm:inv_L} that relates $L$ to its block matrix form involving $X$, hence is applicable to other instances when one does not have such strong structure information of the matrix.

For a prime $p$, we write $p^\alpha\,||\,n$ if $p^{\alpha}\,|\, n$ but $p^{\alpha+1}\!\nmid\! n$.

We start with a theorem due to Lorenzini that relates the eigenvalues of an integer matrix and its invariant factors.

\begin{theorem} \cite{Lorenzini} \label{thm:Lorenzini}
Let $M\neq 0$ be a $n\times n$ diagonalizable integer matrix and let $\lambda_1,\ldots,\lambda_t$ be the distinct non-zero eigenvalues of $M$. Then every nonzero invariant factor of $M$ divides $\prod \lambda_i$.
\end{theorem}

In particular, for Adinkras, we have the following bound:

\begin{proposition} \label{prop:LinvFineq} For an Adinkra on $N$ colors and any prime $p$ dividing $N(N-1)$, the number of invariant factors of $L$ that are divisible by $p$ is at least $\#V/2$.
\end{proposition}

\begin{proof}
Let the $i$th invariant factor be written as $p^{\beta_i}\gamma_i$ where $p$ does not divide $\gamma_i$.  Let $B$ be the set of $i$ for which $\beta_i>0$.

By Theorem~\ref{thm:Lorenzini}, each invariant factor of $L$ divides $(N+\sqrt{N})(N-\sqrt{N})=N(N-1)$.  Suppose $p^\alpha\,||\,N(N-1)$.  Then each $\beta_i\le\alpha$.  We will also have $p^{\alpha\#V/2}\,||\, det(L)$, which is the product of the invariant factors, so that 
\[\alpha\#V/2=\sum_B \beta_i\le (\#B)(\alpha).\]
Therefore $\#B\ge \#V/2$.
\end{proof}

\begin{corollary} \label{coro:pN_invf}
Let $p$ be an odd prime dividing $N$ and suppose $p^\alpha\,||\,N$. Then the $p$-component of the invariant factors of $L$ is $(1^{\#V/2},p^{\alpha\#V/2})$.
\end{corollary}

\begin{proof}
Proposition~\ref{prop:LinvFineq} says that the number of invariant factors of $L$ that are divisible by $p$ is at least $\#V/2$.  The Odd Prime Theorem~\ref{thm:oddprime} says that this is at most $\#V/2$.  Thus, the inequalities in the proof of Proposition~\ref{prop:LinvFineq} are equalities throughout.
\end{proof}

{\sc Proof of Theorem~\ref{thm:main_invf} (for $L$):}
By Theorem~\ref{thm:Lorenzini}, every invariant factor of $L$ divides $N(N-1)$.
The first $\#V/2$ invariant factors of $L$ are relatively prime to $N-1$ by the Elementary Divisor Theorem and the fact that the minor corresponding to the submatrix $NI_{\#V/2}$ is relatively prime to $N-1$.
Furthermore, the first $\#V/2$ invariant factors are relatively prime to $p$ for every odd prime $p|N$ by Corollary~\ref{coro:pN_invf}, so they must be powers of 2 that divide $N$, proving the first assertion.

Next, denote by ${\mathfrak o}(n)$ the largest odd number dividing an integer $n$ and write $N':={\mathfrak o}(N^2-N)$. Then ${\mathfrak o}(f)=1$ for each invariant factor from the first half, and ${\mathfrak o}(f)|N'$ for each of the rest, but the product of all ${\mathfrak o}(f)$'s equals ${\mathfrak o}((N^2-N)^{\#V/2})=N'^{\#V/2}$, which forces  ${\mathfrak o}(f)=N'$ for each invariant factor in the second half. This proves the second assertion. \hfill$\Box$\\

Summarizing, the idea in the above proof is that we can (1) obtain the eigenvalues and determinants of $L$ and $\tilde{L}$ from those of $\hat{L}$ over $\zz[x]$, (2) work over $\zz$ and $\mathbb{C}$ (resp. $\gf_p[x]$) provides an {\em upper bound} (resp. {\em lower bound}) on the power of $p$ (resp. $x-1$) in the invariant factors of $L$ (resp. $\tilde{L}$), hence a {\em lower bound} (resp. {\em upper bound}) of the number of invariant factors divisible by $p$ (resp. $x-1$), and (3) combine the information from (2) over $\gf_p$ to give exact information over $\zz$.
The interaction of different Laplacians over various rings is quite aesthetic in the case of Adinkras.

\section{The 2-rank of the Laplacian, signed and unsigned}
\label{sec:tworank}
In $\gft$, $1=-1$, and so the signed Laplacian modulo $2$, $\overline{L}$, is independent of the signs in the graph, and thus the mod $2$ versions of the signed Laplacian and of the unsigned Laplacian are the same.  This allows us to compare our results to results in the existing literature about the unsigned Laplacian on graphs of the type we are considering in this paper.

Since $\gft$ is a field, the rank of a matrix over $\gft$ is the only invariant of the matrix, up to equivalence by row and column operations.  By Lemma~\ref{lem:XvsTildeX}, the rank of $\overline{L}$ in $\gft$ is the number of odd invariant factors of $L$.  By Corollary~\ref{cor:lasteven}, the last half of the invariant factors for $L$ are even.  By Theorem~\ref{thm:main_invf}, the first half of the invariant factors for $L$ are all powers of $2$.  Therefore, the rank of $\overline{L}$ is equal to the multiplicity of $1$ as an invariant factor of $L$.

In the unsigned case, the Laplacian always has a nontrivial kernel.  For connected graphs, this kernel has dimension $1$, and corresponds to constant functions on the set of vertices.  Therefore the last invariant factor for the unsigned Laplacian is always $0$.  In that case, the invariant factors for the critical group is defined to be the invariant factors of the Laplacian except for this last $0$.  In our signed case, on the other hand, we have seen in Theorem~\ref{cor:eigL} that the eigenvalues of the signed Laplacian are $N\pm\sqrt{N}$, so that if $N\ge 2$, there is no kernel.

When an ordinary graph is an $N$-cube or more generally a Cayley graph of $\gft^r$, the odd component of its critical group can be deduced using representation theoretic techniques, see \cite[Theorem~1.2]{Bai_hypercube} and \cite[Proposition~1.8]{GKM_Cayley}.
However, the $2$-Sylow subgroup of their critical groups, or even the just $2$-rank of the Laplacians, has been rather difficult to understand.
Hence, the known results on these graphs are somewhat parallel to ours on Adinkras, but with some differences in exact statements and machinery.

For example, Reiner conjectured that the critical group of the $N$-cube has exactly $2^{N-1}-1$ invariant factors, which was first proved by Bai in \cite{Bai_hypercube} by an inductive argument.
Using our work, we can give a short proof of it: consider a cubical Adinkra on $N$ colors, the $2$-corank of its Laplacian is $2^{N-1}$ by Corollary~\ref{coro:cubic_Adinkra}, so the critical group of the $N$-cubes has exactly $2^{N-1}-1$ even invariant factors; on the other hand, the $2^{N-1}\times 2^{N-1}$ submatrix $-I$ (up to permutation of row) in the Laplacian together with the Elementary Divisor Theorem implies that the first $2^{N-1}$ invariant factors of it are ones.

The recent work of Gao--Marx-Kuo--McDonald considers Cayley graphs of $\gft^r$.  As we will see shortly, these are quotients of cubes, and so these are the underlying unsigned versions of connected Adinkras.

\begin{definition}
The Cayley graph $\Cayley(G,g_1,\ldots,g_k)$ for a finitely generated group $G$ and a set $\{g_1,\ldots,g_k\}$ of generators of $G$ is a graph whose vertices are the elements of $G$ and where for every vertex $h\in G$, and every generator $g_i$, there is an edge from $h$ to $g_ih$, labeled with $i$.  If $g_i{}^2$ is the identity, then we do not need to indicate a direction on the edge, but more generally we orient the edge from $h$ to $g_ih$.
\end{definition}

If $\evec_1,\ldots,\evec_r$ are the standard basis elements of $\gft^r$, then the colored $r$-cube is $\Cayley(\gft^r,\evec_1,\ldots,\evec_r)$, where the colors of the $r$-cube correspond to the labels of the Cayley graph.  Since all elements of $\gft^r$ are of order 1 or 2 over vector addition, the Cayley graph is not directed.

More generally, we will consider $\Cayley(\gft^r,\mathbf{g}_1,\ldots,\mathbf{g}_N)$.  It is convenient to collect the generators $\mathbf{g}_1,\ldots,\mathbf{g}_N$ as columns of an $r\times N$ matrix $M$.  Then we write this Cayley graph as $\Cayley(\gft^r,M)$.  The columns of $M$ generate $\gft^r$ if and only if the rank of $M$ is $r$ (which in turn guarantees $N\ge r$).

Then Gao--Marx-Kuo--McDonald counts Sylow-$2$ cyclic factors of the critical group of the Cayley graph in a large set of cases:
\begin{definition}
A matrix $M$ is called \emph{generic} if the mod $2$ sum of its columns is non-zero.
\end{definition}

\begin{theorem}[Theorem 1.1 from \cite{GKM_Cayley}]
If $M$ is generic, then the number of Sylow-$2$ cyclic factors of $K(\Cayley(\gft^r,M))$ is $2^{r-1}-1$.
\end{theorem}

We first make precise of the relation between Cayley graphs and quotients of $N$-cubes.

\begin{proposition}
Let $r\le N$ and let $M$ be an $r\times N$ matrix with values in $\gft$ with rank $r$.  Then $\Cayley{(\gft^r,M)}$ is isomorphic to the quotient of the $N$-cube by $\ker(M)$.
\end{proposition}

\begin{proof}
We construct the isomorphism in three steps.  First, we do row operations and column permutations on $M$.  Row operations on $M$ correspond to linear isomorphisms of $\gft^r$, and thus, give rise to isomorphic Cayley graphs.  Permutations of the columns of $M$, corresponds to relabeling the generators.  If $M$ has rank $r$, then using row reduction and permuting columns, we can turn $M$ into the block form
\begin{equation}
    M=\begin{bmatrix}
        I&A
    \end{bmatrix}
    \label{eqn:mref}
\end{equation}
where $I$ is the $r\times r$ identity matrix.  Write the columns of $A$ as $\mathbf{a}_1,\ldots,\mathbf{a}_{N-r}$.  The Cayley graph is then an $r$-cube, together with edges labeled $r+i$ connecting each vertex $\mathbf{x}$ to $\mathbf{x}+\mathbf{a}_i$.  Indeed, Example~\ref{ex:k4} is drawn initially in this form: a 3-cube with diagonals drawn using $\vec{a}_1=(1,1,1)$.

Thus we have an isomorphism of the Cayley graph that may involve permuting the edge labels.  This is our first isomorphism, and for the remainder of the proof, we can assume $M$ is of the form (\ref{eqn:mref}).

Second, consider the map
\[Z\colon \gft^r\to \gft^N\]
\[Z(x_1,\ldots,x_r)=(x_1,\ldots,x_r,0,\ldots,0)\]
that appends $N-r$ copies of $0$ at the end of the vectors.

Third, consider the map
\[\pi\colon \gft^N\to \gft^N/\ker M\]
\[\pi(v)=[v]\]
that sends each vertex of the $N$-cube to its corresponding coset.

The desired isomorphism on vertices is the composition $\pi\circ Z$.

First, we prove that $\gft^N$ is the direct sum of the image of $Z$ and the kernel of $M$.  To see this, consider any element of their intersection.  It must be of the form
\[(u_1,\ldots,u_r,0,\ldots,0).\]
If this is in the kernel of $M$ in the form (\ref{eqn:mref}), we must have
\[
\begin{bmatrix}I&A\end{bmatrix}
\begin{bmatrix}u_1\\\vdots\\u_r\\0\\\vdots\\0\end{bmatrix}
=\begin{bmatrix}0\\\vdots\\0\end{bmatrix}
\]
which means $u_1=\cdots=u_r=0$.  Therefore the intersection of the image of $Z$ and the kernel of $M$ is trivial.

The dimension of the image of $Z$ is $r$, and the dimension of the kernel of $M$ is $N-r$.  Thus, $\gft^N$ is the direct sum of the image of $Z$ and the kernel of $M$, and the composition $\pi\circ Z$ is a linear isomorphism.  In particular, it is a bijection of vertices.

We now consider the edges in $\Cayley(\gft^r,M)$ and compare them with the edges in $\gft^N/\ker M$.  There are two kinds of edges in $\Cayley{(\gft^r,M)}$: edges of the type $\evec_i$ for $1\le i\le r$, and edges of the type $\mathbf{a}_i$ for $1\le i\le N-r$.  If two vertices $v, w$ in $\gft^r$ are connected by an edge of type $\evec_i$, then $Z(v)$ and $Z(w)$ differ in coordinate $i$ in $\gft^N$, and are connected by an edge of color $i$.  Thus, $[Z(v)]$ and $[Z(w)]$ are connected by an edge of color $i$.

If $v$ and $w$ in $\gft^r$ are connected by an edge of type $\mathbf{a}_i$; i.e., $v+\mathbf{a}_i=w$ in $\gft^r$.  In $\gft^N$, then, $Z(v)+Z(\mathbf{a}_i)=Z(w)$.  There is an edge of color $r+i$ from $Z(v)$ to $Z(v)+\evec_{r+i}$.

We now show that
\[Z(\mathbf{a}_i)+\evec_{r+i}\in\ker M.\]
Writing out the matrix multiplication in block matrix form, we get:
\[\begin{bmatrix}I&A\end{bmatrix}\begin{bmatrix}\mathbf{a}_i\\\evec{i}\end{bmatrix}
=\begin{bmatrix}\mathbf{a}_i+\mathbf{a}_i\end{bmatrix}
=\begin{bmatrix}0\\\vdots\\0\end{bmatrix}
\]
Therefore, $Z(v)+\evec_{r+i}$ is in the same coset of $\ker M$ as $Z(v)+Z(\mathbf{a}_i)=Z(w)$.  Hence, in the quotient $\gft^N/\ker M$, there is an edge of color $r+i$ from $\pi(Z(v))$ to $\pi(Z(w))$.
\end{proof}

This procedure can be reversed in the following way: given a linear code $C$ with generating matrix $G$ of size $k\times N$, row operations on $G$ do not change $C$, and row reduction results in $G$ being in reduced row echelon form.  Applying permutations of columns of $G$, corresponding to permuting the colors, can be done to write $G$ in the form
\[G=\begin{bmatrix}
A^T&I
\end{bmatrix}.
\]
where $A$ is an $(N-k)\times k$ matrix.  Writing
\[M=\begin{bmatrix}
I&A
\end{bmatrix}
\]
it is straightforward to check that the kernel of $M$ is the row space of $G$, i.e., the code $C$.

Thus, the work of Gao--Marx-Kuo--MacDonald \cite{GKM_Cayley} applies to connected Adinkras.  To translate the language between the two areas, we prove some propositions.

\begin{proposition}
The matrix $M$ is generic if and only if the code $C$ does not contain the all-ones vector $\ones$.
\label{prop:genericones}
\end{proposition}
\begin{proof}
The mod $2$ sum of the columns of $M$ is $M\ones$, that is, the product of the matrix $M$ and the all-ones column vector $\ones$.  This is the zero vector if and only if $\ones\in\ker(M)$.
\end{proof}

\begin{proposition}
Suppose $G$ is a connected graph.  The number of Sylow-$2$ cyclic factors of $K(G)$ is equal to one less than the number of even invariant factors of the unsigned Laplacian of $G$.\label{prop:countevenfactors}
\end{proposition}
\begin{proof}
If the invariant factors of the unsigned Laplacian are
\[x_1,\ldots,x_{\#V}\]
with $x_1\,|\, \cdots \,|\,x_{\#V}$, then $x_{\#V}=0$ and the other invariant factors are non-zero (since $G$ is connected).  Then
\[K(G)\cong \zz/x_1\zz\oplus\cdots\oplus \zz/x_{\#V-1}\zz.\]
The Sylow-$2$-cyclic factors are the terms $\zz/x_i\zz$ where $x_i$ is even.
\end{proof}

\begin{corollary}\label{cor:notallone}
If the underlying graph of an Adinkra is a $N$-cube quotient by a doubly even code $C$ that does not contain $\ones$, then the invariant factors of the Adinkra are $(1^{\#V/2},(N^2-N)^{\#V/2})$.
\end{corollary}

\begin{proof}
Let $k$ be the dimension of $C$.  Then the number of vertices in the Adinkra is $\#V=2^{N-k}$.  Let $M$ be the $(N-k)\times N$ matrix with $\ker(M)=C$.  Then $\Cayley(\gft^{N-k},M)$ is isomorphic to the underlying unsigned graph of the Adinkra.

By Proposition~\ref{prop:genericones}, $\ones \not\in C$ implies that $M$ is generic, and then Theorem 1.1 in \cite{GKM_Cayley} says that the number of Sylow-$2$ cyclic factors of $K(G)$ is $2^{N-k-1}-1$.  Then by Proposition~\ref{prop:countevenfactors}, the number of even invariant factors for the unsigned Laplacian for the Adinkra is $2^{N-k-1}=\#V/2$.  Thus, the first $\#V/2$ invariant factors for the unsigned Laplacian (and thus, for the signed Laplacian) are odd.

By Theorem~\ref{thm:main_invf}, the first $\#V/2$ invariant factors are all 1, and the remaining $\#V/2$ invariant factors are $N(N-1)$.
\end{proof}

Theorem~\ref{thm:prism} and Corollary~\ref{coro:multi4} are special cases of this, though they were proved using special properties of Adinkras as signed graphs; the converse of the statement is the equality assertion of Conjecture~6.1 of \cite{GKM_Cayley}, which is still open (but our data provide further evidence of it).
On the other hand, we can use our result to prove a special case of the inequality assertion of their Conjecture 6.1 by considering an arbitrary Adinkra whose underlying graph is the Cayley graph, and apply either Theorem~\ref{thm:inv_L} or Proposition~\ref{prop:LinvFineq}.

\begin{proposition}
If $\ker(M)$ is a doubly even code, then there are at least $2^{N-k-1}-1$ even invariant factors of the critical group of $\Cayley{(\gft^{N-k},M)}$.
\end{proposition}

Similarly, Conjecture~6.2 of \cite{GKM_Cayley} hypothesizes that the number of even invariant factors of the critical group of $\Cayley{(\gft^{N-k},M)}$ is odd if and only if the $2$-adic valuations of the nonzero eigenvalues of its Laplacian are not the same.
Theorem~\ref{thm:inv_L} implies that whenever $\ker(M)$ is a doubly even code, the number of even invariant factors is always odd, so we either have given a proof of the special case, or have provided a counterexample, depending on the eigenvalues of the corresponding Cayley graphs.

\section{Independence of signatures}
\label{sec:dependsig}
So far we have focused on the role of the topology of the Adinkra, or equivalently, its code, on the invariant factors of the signed Laplacian.  For instance, in Section~\ref{sec:table} we provided a list of invariant factors for Adinkras, identifying each according to its code, with no reference to the signature.  But a given graph might have many totally odd signatures.  Nevertheless, based on some preliminary experimental findings, we propose:
\begin{conjecture}
The invariant factors do not depend on the signature of the Adinkra.
\end{conjecture}

So far, we have theorems (Theorem~\ref{thm:prism}, Corollary~\ref{coro:multi4}, and Corollary~\ref{cor:notallone}) that say that the invariant factors are $(1^{\#V/2},(N(N-1))^{\#V/2})$ in a wide range of cases, such as when $N$ is not a multiple of $4$, or when the graph is a prism, or when the all-one codeword $\ones$ is not in the code for the Adinkra; and so changing the signature in those cases would not change the invariant factors.

In addition, there is an operation among totally odd signatures called \emph{vertex switching} which, as we will see, does not affect the invariant factors.  We review the basic definitions and facts about vertex switching from \cite{zaslavsky}.  We will initially state the definition for any signed graph, but our interest will be for Adinkras.

\begin{definition}
If $G$ is a signed graph and $v$ is a vertex of $G$, then the result of \emph{switching at $v$} is a graph $G_v$ whose vertices and edges are the same as $G$, but with sign function $S_v(\mu)$ defined as:
\[S_v(\mu)(e)=\begin{cases}
\mu(e),&\mbox{if $e$ is not incident to $v$, and}\\
-\mu(e),&\mbox{if $e$ is incident to $v$.}
\end{cases}
\]

Switchings at various vertices commute, and so compositions of such are determined by which vertices we switched at.  Thus, we can generalize the notion of switching to \emph{switching at a set of vertices}, as follows: if $W$ is a set of vertices of $G$, define the signature $S_W(\mu)$ to be:
\[S_W(\mu)(e)=\begin{cases}
\mu(e),&\mbox{if $e$ is incident to either 0 or 2 vertices in $W$, and}\\
-\mu(e),&\mbox{if $e$ is incident to exactly 1 vertex in $W$.}
\end{cases}
\]
\end{definition}

\begin{example}
For instance, if we take the following signed graph:
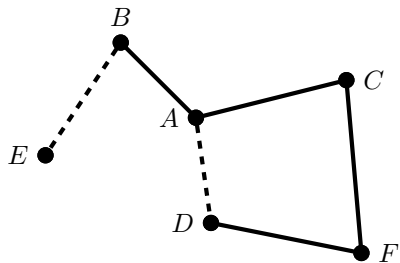
\begin{center}
\begin{tikzpicture}
\GraphInit[vstyle=Welsh]
\SetVertexNormal[MinSize=5pt]
\SetUpEdge[labelstyle={draw},style={ultra thick}]
\tikzset{Dash/.style={dashed,draw,ultra thick}}
\Vertex[x=0,y=0,Math,L={A},Lpos=180]{A}
\Vertex[x=-1,y=1,Math,L={B},Lpos=90]{B}
\Vertex[x=2,y=.5,Math,L={C},Lpos=0]{C}
\Vertex[x=.2,y=-1.4,Math,L={D},Lpos=180]{D}
\Vertex[x=-2,y=-.5,Math,L={E},Lpos=180]{E}
\Vertex[x=2.2,y=-1.8,Math,L={F},Lpos=0]{F}
\AddVertexColor{black}{A,B,C,D,E,F}
\Edge(A)(B)
\Edge(A)(C)
\Edge[style=Dash](A)(D)
\Edge[style=Dash](B)(E)
\Edge(C)(F)
\Edge(D)(F)
\end{tikzpicture}
\end{center}
the effect of switching at the vertex $A$ is the following signed graph:
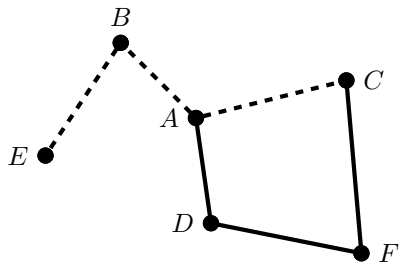
\begin{center}
\begin{tikzpicture}
\GraphInit[vstyle=Welsh]
\SetVertexNormal[MinSize=5pt]
\SetUpEdge[labelstyle={draw},style={ultra thick}]
\tikzset{Dash/.style={dashed,draw,ultra thick}}
\Vertex[x=0,y=0,Math,L={A},Lpos=180]{A}
\Vertex[x=-1,y=1,Math,L={B},Lpos=90]{B}
\Vertex[x=2,y=.5,Math,L={C},Lpos=0]{C}
\Vertex[x=.2,y=-1.4,Math,L={D},Lpos=180]{D}
\Vertex[x=-2,y=-.5,Math,L={E},Lpos=180]{E}
\Vertex[x=2.2,y=-1.8,Math,L={F},Lpos=0]{F}
\AddVertexColor{black}{A,B,C,D,E,F}
\Edge[style=Dash](A)(B)
\Edge[style=Dash](A)(C)
\Edge(A)(D)
\Edge[style=Dash](B)(E)
\Edge(C)(F)
\Edge(D)(F)
\end{tikzpicture}
\end{center}
The edges incident to $A$ have all switched signs, and the other edges are unaffected.

If we take this resulting signature and furthermore switch at $C$, we get the following:
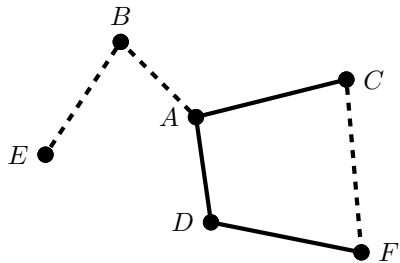
\begin{center}
\begin{tikzpicture}
\GraphInit[vstyle=Welsh]
\SetVertexNormal[MinSize=5pt]
\SetUpEdge[labelstyle={draw},style={ultra thick}]
\tikzset{Dash/.style={dashed,draw,ultra thick}}
\Vertex[x=0,y=0,Math,L={A},Lpos=180]{A}
\Vertex[x=-1,y=1,Math,L={B},Lpos=90]{B}
\Vertex[x=2,y=.5,Math,L={C},Lpos=0]{C}
\Vertex[x=.2,y=-1.4,Math,L={D},Lpos=180]{D}
\Vertex[x=-2,y=-.5,Math,L={E},Lpos=180]{E}
\Vertex[x=2.2,y=-1.8,Math,L={F},Lpos=0]{F}
\AddVertexColor{black}{A,B,C,D,E,F}
\Edge[style=Dash](A)(B)
\Edge(A)(C)
\Edge(A)(D)
\Edge[style=Dash](B)(E)
\Edge[style=Dash](C)(F)
\Edge(D)(F)
\end{tikzpicture}
\end{center}
Comparing this to our initial signature, we see that the edge $AC$ is back to solid (what it used to be originally) and edges that switched are the ones that are incident to $A$ or $C$ but not both.
\end{example}

This procedure is interesting in the case of Adinkras because it preserves the totally odd condition that is used in the definition of Adinkras:
\begin{proposition}
Suppose we have an Adinkra with totally odd signature $\mu$.  Let $v$ be a vertex of the Adinkra.  Then $S_v(\mu)$ is also a totally odd signature on the Adinkra.
\end{proposition}
\begin{proof}
Any bicolor cycle that does not contain $v$ is unaffected by the vertex switching.  If $v$ is in the bicolor cycle, the vertex switching changes the sign of two edges of the cycle.  Therefore the number of dashed edges modulo 2 is unchanged.
\end{proof}

Now we show that vertex switches do not change the invariant factors for the signed Laplacian:
\begin{proposition}
Let $G$ be a signed graph with signature $\mu$, and choose an ordering of the vertices of $G$ so that the signed adjacency and Laplacian operators are matrices.

Let $v$ be the $i$th vertex of $G$.  Then the signed adjacency matrix for $S_v(\mu)$ is obtained from the signed adjacency matrix for $\mu$ by multiplying by $-1$ the $i$th row and $i$th column.

Likewise the signed Laplacian for $S_v(\mu)$ is obtained from the signed Laplacian for $\mu$ by the same operation.
\end{proposition}

\begin{proof}
Each $(j,k)$ entry of the signed adjacency matrix is either $0$, $1$, or $-1$, depending on the existence and sign of the edge connecting the $j$th vertex and the $k$th vertex.  If neither $j$ nor $k$ is $i$, then this is unaffected by a vertex switch.  If $j=k=i$, then this is unaffected as well.  But if $j=i$ and $k\not=i$ or vice-versa, then this reverses sign.

The same argument works for the signed Laplacian.
\end{proof}

Multiplying a row by $-1$ is a row operation, and multiplying a column by $-1$ is a column operation.  These do not affect the invariant factors of a matrix.  Therefore,
\begin{corollary}
Switching vertices does not change the invariant factors of the signed Laplacian.
\end{corollary}

This raises the question as to what other totally odd signatures exist.  The following theorem comes from \cite{doranApplicationCubicalCohomology2017}.
\begin{theorem}
Let $A$ be a connected Adinkra with code $C$.  There is a bijection between the set of totally odd signatures on $A$ modulo vertex switchings and the code $C$.
\end{theorem}

In addition to vertex switchings, we could permute the vertices and edges of the graph in such a way that the Adinkra remains the same, except for the signature:
\begin{definition}
Given an Adinkra, a \emph{color-preserving graph automorphism} $\phi$ of the Adinkra is a pair $(\phi_V,\phi_E)$ where $\phi_V$ is a permutation of the vertex set and $\phi_E$ is a permutation of the edge set of the graph, such that if $e$ is an edge connecting the vertices $v$ and $w$, then $\phi_E(e)$ is an edge connecting the vertices $\phi(v)$ and $\phi(w)$, and $\phi_E(e)$ has the same color as $e$.

Given a color-preserving graph automorphism $\phi$, and a totally odd signature $\mu$ on the Adinkra, we define the \emph{pullback} $\phi^{-1}(\mu)$ to be the signature so that $\phi^{-1}(\mu)(e)=\mu(\phi_E(e))$.
\end{definition}

\begin{proposition}
Let $\phi$ be a color-preserving graph automorphism.  We write $\phi_V$ as a matrix: the $(i,j)$ entry of $P_\phi$ is $1$ if $\phi_V$ applied to the $j$th vertex is the $i$th vertex; and $0$ otherwise.  Then $P_\phi$ is a permutation matrix and the signed adjacency matrix of the Adinkra with the pullback signature is
\[A_\phi=P_\phi{}^{-1} A P_\phi.\]
Likewise, the signed Laplacian of the Adinkra with the pullback signature is
\[L_\phi=P_\phi{}^{-1} L P_\phi.\]
\end{proposition}
\begin{proof}
We examine the $(i,j)$ entry of this matrix.  Write $v_i$ for the $i$th vertex and $v_j$ for the $j$th vertex.  Suppose we write $v_p=\phi_V(v_i)$ and $v_q=\phi_V(v_j)$.  Then the $(i,j)$ entry of $A_\phi$ is the $(p,q)$ entry of $A$.  This is $0$ if there is no edge from $v_p$ to $v_q$, and $\mu(v_p,v_q)$ otherwise.  This is the value of $\phi^{-1}(\mu)$ on the edge from $v_i$ to $v_j$.  Thus,
\[A_\phi=P_\phi{}^{-1} A P_\phi.\]

The formula for the signed Laplacian follows from the definition
\[L=NI-A.\]
\end{proof}
\begin{corollary}
Color-preserving graph automorphisms do not change the invariant factors of the signed Laplacian.
\end{corollary}

Again from \cite{doranApplicationCubicalCohomology2017},
\begin{theorem}
Let $A$ be a connected Adinkra with code $C$.  As before, we write $\ones$ for the string $1\cdots 1$ of length $N$.

If $\ones$ is not a codeword in $C$, then every totally odd signature can be obtained from any other by a composition of a color-preserving graph automorphism and a number of vertex switchings.

If $\ones$ is a codeword in $C$, then the action on the set of totally odd signatures by compositions of color-preserving graph automorphisms and vertex switchings has two orbits.
\end{theorem}

So when $\ones$ is not a codeword in $C$, all totally odd signatures give rise to the same invariant factors.  By Corollary~\ref{cor:notallone}, we already know the invariant factors will be $(1^{\#V/2},(N(N-1))^{\#V/2})$.

When $\ones$ is a codeword in $C$, there are two totally odd signatures to check.  The cases checked so far ($d_4$, $h_8$, $d_4\oplus d_4$, $d_8$, $e_8$) indicate that the two possibilities have signed Laplacians with the same invariant factors.

\section{Other Conjectures and Further Directions}

Besides the independence of dashings discussed above, the most obvious and central question is to determine the $2$-rank of the Laplacian of an Adinkra and the $2$-Sylow subgroups of its critical group.
While we do not have a precise conjecture on the $2$-Sylow subgroups, we state the converse of Corollary~\ref{cor:notallone} (which is a special case of \cite[Conjecture 6.1]{GKM_Cayley}) as our next conjecture.

\begin{conjecture}
If $\ones$ is a codeword of a doubly even code associated to an Adinkra, then the Adinkra has non-trivial invariant factors (resp. the $2$-rank is less than $\#V/2$).
\end{conjecture}

A natural question from the interaction of Smith Normal Forms over different rings in (\ref{eq:CD}) is whether the Smith Normal Forms of $L,\tilde{L}$ over $\zz,\gf_p[x]$ can be deduced from a ``universal'' Smith Normal Form of $\hat{L}$ over $\zz[x]$.
The existence of such a Smith Normal Form is not {\em a priori} as $\zz[x]$ is not a PID.
Indeed, it is possible to prove that if the Adinkra has non-trivial invariant factors, then such universal Smith Normal Form does not exist.
We conjecture the converse is true.

\begin{conjecture} The Smith Normal Form of $\hat{L}$ exists (over $\zz[x]$) whenever the Smith Normal Form of $L$ is $(1^{\#V/2},(N^2-N)^{\#V/2})$, in which case it must take the form $(1^{\#V/2},(2(N-1)x+(N-1)(N-2))^{\#V/2})$.
\end{conjecture}

More generally, we believe the algebraic techniques developed in Section~\ref{sec:oddprime} could be useful in other settings.
It is often the case that a good upper bound of the $p$-rank of an integer matrix can be obtained by working over $\zz$.
For example, if a (diagonalizable) matrix only has a few distinct eigenvalues that are small, then Theorem~\ref{thm:Lorenzini} gives a non-trivial upper bound of the $p$-rank; another example in the literature is strongly regular graphs, see \cite{BV_RankSRG}.
In contrast, it seems that there are less techniques available for giving good lower bounds of $p$-rank; our ``lift to $\zz[x]$'' method provides a general approach for obtaining lower bounds, as the following proposition shows:

\begin{proposition} \label{prop:uni_bound}
Let $M\in\zz^{n\times n}$ and $p$ be a prime.
Then the $p$-corank of $M$ equals the minimum of the degree of $x-1$ in $\overline{\det\hat{M}}\in\gf_p[x]$ over all non-singular $\hat{M}\in\zz[x]^{n\times n}$ such that $\hat{M}|_{x=1}=M$.
\end{proposition}

\begin{proof}
Following Section~\ref{subsec:oddprime}, the minimum is at least the $p$-corank, so it remains to show that equality is attained. 
Let ${\rm diag}(d_1,\ldots,d_n)=BMC$ be the Smith Normal Form of $M$, where $p|d_i$ for $i>k$.
Set $\hat{M}=B^{-1}{\rm diag}(d_1,\ldots,d_k,x-(1-d_{k+1}),\ldots,x-(1-d_n))C^{-1}$ (since $\det(B),\det(C)=\pm 1$, $B^{-1}, C^{-1}$ are integral matrices).
Then $\hat{M}|_{x=1}=M$, and $\det(\hat{M})=\pm\prod_{i\leq k}d_i\prod_{j>k} (x-(1-d_j))$, whose reduction over $\gf_p[x]$ equals  $(x-1)^{n-k}$ up to non-zero scalar multiple.
\end{proof}

Finding lifts of special matrices to $\zz[x]$ that give good lower bounds (or exact values) of their $p$-rank, without computing the Smith Normal Form directly, seems to be an interesting problem.  In our case, we used the colors of the edges to correspond to the indeterminates $x_i$.  Likewise, one can ask whether other natural lifts of special matrices and their determinants (for instance, natural $q$-analogs of combinatorial matrices) give exact $p$-ranks for different $p$'s in the sense of Proposition~\ref{prop:uni_bound}.

Finally, another interesting direction is to explore is the connection between the critical groups of Adinkras and representation theory.
The critical groups of ``algebraic'' (ordinary) graphs such as Cayley graphs of abelian groups \cite{DJ_Cayley,CSX_Paley,TS_Polar}, or more generally integer matrices coming from representation theory such as McKay-Cartan matrices \cite{Gaetz, BKR_ChipFiring,GHR_Hopf}, can often be studied and interpreted using representation theoretical techniques.
In view of the connections with supersymmetry algebras and Clifford algebras, it would be illuminating to understand the results in this paper using the representation theory of these objects.

\section{Acknowledgments}
KI was partially supported by the endowment of the Ford Foundation Professorship of Physics at Brown University, and by the U.S. National Science Foundation grant PHY-1315155.
CHY was supported by Croucher Fellowship for Postdoctoral Research and Trond Mohn
Foundation project ``Algebraic and Topological Cycles in Complex and Tropical
Geometries'' during his affiliation to Brown University and University of Oslo, respectively.

\bibliographystyle{amsplain}
\bibliography{refs}

\end{document}